\documentclass[a4paper,11pt,oneside]{article}
\usepackage[utf8]{inputenc}
\usepackage{graphicx}
\usepackage{algorithm, algo_olm}
\usepackage{amsmath,amssymb}
\usepackage{amsthm, mathtools}
\usepackage{amsfonts}
\newtheorem{thm}{Theorem}[section]

\newtheorem{prop}[thm]{Proposition}

\newtheorem{lem}[thm]{Lemma}
\newtheorem{rem}[thm]{Remark}

\begin{document}

\title{A nonintrusive Reduced Basis Method applied to aeroacoustic simulations}
\maketitle

\begin{center}
{Fabien Casenave$^{1}$, Alexandre Ern$^{1}$, and Tony Leli\`{e}vre$^{1,2}$}\\
\end{center}
\begin{center}
\end{center}
\begin{center}
$^1$ Universit\'{e} Paris-Est, CERMICS (ENPC), 6-8 Avenue Blaise Pascal, Cit\'{e} Descartes ,  F-77455 Marne-la-Vall\'{e}e, France
\end{center}
\begin{center}
$^2$ INRIA Rocquencourt, MICMAC Team-Project, Domaine de Voluceau, B.P. 105, 78153 Le Chesnay Cedex, France 
\end{center}

\begin{abstract}
The Reduced Basis Method can be exploited in an efficient way only if the so-called affine dependence assumption on the operator and
right-hand side of the considered problem with respect to the parameters is satisfied.
When it is not, the Empirical Interpolation Method is usually used to recover this assumption approximately.
In both cases, the Reduced Basis Method requires to access and modify the assembly routines of the corresponding computational code,
leading to an intrusive procedure.
In this work, we derive variants of the EIM algorithm and explain how they can be used to
turn the Reduced Basis Method into a nonintrusive procedure.
We present examples of aeroacoustic problems solved by integral equations 
and show how our algorithms can benefit from the linear algebra tools available in the considered code.
\end{abstract}

\section{Introduction}
\label{intro}

In many problems such as optimization, uncertainty propagation, and real-time simulations, one needs
to solve a parametrized problem for many values of the parameters. Among the
various available methods to reduce the computational cost, the Reduced Basis Method (RBM) has received
increased interest over the last decade
(see \cite{Machiels,RB,Maday,prud2} for a detailed presentation and \cite{RBconv2} for some convergence results). 
Consider the following problem: Find $u_\mu\in\mathcal{V}$ such that
\begin{equation}
\label{eq:fullpb}
a_\mu(u_\mu,v)=c_\mu(v),\qquad\forall v\in\mathcal{V},
\end{equation}
where $\mu\in\mathcal{P}$ is the parameter, $a_\mu$ is a sesquilinear form, $c_\mu$ is a linear form, and $\mathcal{V}$ is a finite-dimensional functional space of size $n$,
where $n$ is typically very large.
Since the linear problem~\eqref{eq:fullpb} is written on a finite-dimensional space, we can consider the following matrix form:
\begin{equation}
\label{eq:fullpbmat}
A_\mu U_\mu = C_\mu,
\end{equation}
where $A_\mu\in\mathbb{C}^{n\times n}$ and $C_\mu\in\mathbb{C}^{n}$.
We refer to the solutions to~\eqref{eq:fullpb} as truth solutions.

The RBM allows one to compute very fast an approximation of the truth solution $u_\mu$ by means of an offline/online procedure.
The online stage is a Galerkin procedure written on a basis of
so-called truth solutions $u_{\mu_l}$, $1\leq l\leq \hat{n}\ll n$, rather than on a basis of $\mathcal{V}$.
The parameter values $\mu_l$ are selected by a greedy algorithm in the offline stage, where the functions $u_{\mu_{l}}$ of
the reduced basis are also precomputed.
Denote by $U$ the rectangular matrix of size $n\times\hat{n}$ such that $(U)_{i,l}=\gamma_i(\mu_l)$, where $\gamma_i(\mu_l)$, $1\leq i\leq n$, are the
coefficients of $u_{\mu_l}$ on the basis of $\mathcal{V}$.
Then, the RBM approximation is computed by solving the reduced problem $\hat{A}_\mu \hat{\gamma}(\mu)=\hat{C}_\mu$, where $\hat{A}_\mu=U^t A_\mu U$ and $\hat{C}_\mu=U^t C_\mu$,
so that 
\begin{equation*}
\hat{u}_\mu(x):=\sum_{l=1}^{\hat{n}}\hat{\gamma}_l(\mu)u_{\mu_l}(x)\approx u_{\mu}(x).  
\end{equation*}

The efficiency of the RBM hinges on the assumption of an affine dependence of the operator and the right-hand side with respect to the parameter.
This assumption states that
\begin{equation}
\label{eq:affdep}
(A_\mu)_{i,j}=\sum_{m=1}^d \alpha_m(\mu)(A_m)_{i,j},\qquad 1\leq i,j\leq n,
\end{equation}
where $A_m$ denote parameter independent matrices, and $\alpha_m$ are complex-valued functions of the parameter.
We only discuss the case of the operator $A_\mu$, the right-hand side $C_\mu$ being treated in the same way.
Owing to the separated representation~\eqref{eq:affdep}, the assembly of the reduced problems
and the computation of the a posteriori error bound are performed in complexity independent of $n$
(see~\cite{M2AN} for the computation of the error bound).
It consists in precomputing the matrices $\hat{A}_m=U^tA_mU$ in the offline stage, and then considering
$\hat{A}_\mu=\sum_{m=1}^d \alpha_m(\mu)\hat{A}_m$ in the online stage.
However, this procedure requires in general nontrivial modifications of the
assembling routines of the computational code since various terms of the variational formulation at hand corresponding to the
matrices $A_m$ in~\eqref{eq:affdep} have to be accessed separately.
This issue can be readily dealt with in simple cases. For instance, consider the model problem~\eqref{eq:fullpb}
where 
\begin{equation*}
\displaystyle a_\mu(u_\mu,v)=\int_{\Omega} {\nabla}u_\mu(x)\cdot {\nabla}v(x)dx+\mu \int_{\Omega}u_\mu(x)v(x)dx,
\end{equation*}
and $\mathcal{P}$ is one-dimensional.
Define $A_0$, $A_1\in\mathbb{R}^{n\times n}$ by 
\begin{equation*}
(A_0)_{i,j}=\int_{\Omega} {\nabla}\varphi_i(x)\cdot {\nabla}\varphi_j(x)dx, \qquad 1\leq i,j\leq n,
\end{equation*}
and
\begin{equation*}
(A_1)_{i,j}=\int_{\Omega}\varphi_i(x)\varphi_j(x)dx,\qquad 1\leq i,j\leq n,
\end{equation*}
where $\varphi_i(x)$ are the finite element
basis functions, so that
\begin{equation}
\label{eq:formulaAmu}
A_\mu=A_0+\mu A_1.
\end{equation}
Taking two values $\mu_1\neq \mu_2$ of the parameter,~\eqref{eq:formulaAmu} can be rewritten as
\begin{equation}
\label{eq:formuleAmunonint}
A_\mu=\frac{\mu_2-\mu}{\mu_2-\mu_1}A_{\mu_1}+\frac{\mu-\mu_1}{\mu_2-\mu_1}A_{\mu_2},
\end{equation}
which still has an affine dependence with respect to the parameter.
In~\eqref{eq:formuleAmunonint}, we only require to evaluate $A_\mu$ for some values of $\mu$. Since the matrix $A_\mu$
is the result of
the whole assembly procedure and is therefore easily accessed in computational codes, the formula~\eqref{eq:formuleAmunonint}
is called nonintrusive.
More generally, we say that a formula to compute $A_\mu$ or $C_\mu$ is nonintrusive if it only requires to access the whole matrix or
the right-hand side for some selected values of the parameter $\mu$.
The first goal of this work is to extend the idea leading to~\eqref{eq:formuleAmunonint} to more complicated parameter dependencies,
and to apply it to obtain nonintrusive
formulae for the matrix and right-hand side of~\eqref{eq:fullpbmat} with an affine dependence on the parameter.

The second objective is to develop a nonintrusive procedure to get approximate affine representations of the operator and right-hand
side, when affine dependence does not hold. In this case, the Empirical Interpolation Method (EIM) can be used.
In this work, we present variants of the classical EIM algorithm, and,
to the price of an additional EIM approximation the accuracy of which we can control, we derive, in a quite
general framework, nonintrusive approximations of, say, the system matrix in the form
\begin{equation}
\label{eq:goal}
A_\mu\approx\sum_{m=1}^r \beta_m(\mu)A_{\mu_m},
\end{equation}
where $\mu_m$, $1\leq m\leq r$, are some selected values of the parameter (which are different from the parameter values $\mu_l$,
$1\leq l\leq \hat{n}$, selected by the greedy algorithm in the offline stage of the RBM), and where $\beta_m(\mu)$ can be
computed efficiently (namely with a complexity independent of $n$).

The article is organized as follows.
In Section~\ref{sec:synthEIM}, we briefly recall the classical EIM algorithm, and present some variants that are
useful in the present context. Then, nonintrusive procedures to approximate $A_\mu$ are derived in Section~\ref{sec:nonint}.
In Section~\ref{sec:affdepavailable}, we consider the case where affine dependence is already available, and in
Section~\ref{sec:nonaffdep} the general case.
Finally, numerical simulations are presented on aeroacoustic problems solved by integral equations in Section~\ref{sec:manyqueries},
where the use of the nonintrusive formulae is crucial.

\section{Classical EIM and variants}
\label{sec:synthEIM}

Consider a function $g(\mu,x)$ defined over $\mathcal{P}\times\Omega$ for two sets $\mathcal{P}$ and $\Omega$.
We look for an approximation of this function in a separated form with respect to $\mu$ and $x$.
There are different possible ways to achieve such an approximation using EIM-like algorithms.
An EIM algorithm consists of an offline stage, where some quantities are precomputed within a greedy procedure, and an online stage
where the approximation is computed making use of these precomputed quantities.

\subsection{{\em Slice 1}}
\label{sec:slice1}

First, we recall the classical EIM as introduced in~\cite{Barrault}, see also \cite{Maday}. We denote the offline
stage of this algorithm by
${\rm EIM}^{\rm S1}$, S1 refering to {\em Slice 1}, since the first variable is treated before the second variable in the
construction.
Fix an integer $d>1$ (the total number of interpolation points). For all $1\leq k \leq d$, the rank-$k$ approximation operator
$I_{k}^{\rm S1}$ is defined as
\begin{equation}
\label{eq:onlinea}
\left(I_{k}^{\rm S1} g\right)(\mu,x) := \sum_{m=1}^{k} \lambda^{\rm S1}_m(\mu) q^{\rm S1}_{m}(x),
\end{equation}
where the functions $\lambda_m^{\rm S1}(\mu)$, $1\leq m\leq k$, solve the linear system
\begin{equation}
\label{eq:onlineapb}
\sum_{m=1}^{k}B^{\rm S1}_{l,m}{\lambda}^{\rm S1}_m(\mu)=g(\mu,x^{\rm S1}_l), \qquad \forall 1\leq l\leq k.
\end{equation}
The functions $q^{\rm S1}_{m}(\cdot)$ and the matrices $B^{\rm S1}\in\mathbb{R}^{k\times k}$, which are
lower triangular with unity diagonal, are constructed as in the offline stage
described in Algorithm~\ref{algo0}, where $\delta^{\rm S1}_k={\rm Id}-I^{\rm S1}_{k}$ and $\|\cdot\|_{\Omega}$ is a norm
on $\Omega$, for instance the $L^\infty\left(\Omega\right)$- or the $L^2\left(\Omega\right)$-norm.
In practice, the argmax appearing in Algorithm~\ref{algo0} is searched over finite subsets of
$\mathcal{P}$ and $\Omega$, denoted respectively by $\mathcal{P}_{\rm trial}$ and $\Omega_{\rm trial}$.
Note that Algorithm~\ref{algo0} also constructs the set of points $\{x_l^{\rm S1}\}_{1\leq l\leq d}$ in $ \Omega$ used in~\eqref{eq:onlineapb},
and a set of points $\{\mu_l^{\rm S1}\}_{1\leq l\leq d}$ in $\mathcal{P}$.
The following assumption is made.
\begin{itemize}
 \item[\bf(H)] The dimension of $\underset{\mu\in\mathcal{P}}{\rm Span}\left(g(\mu,\cdot),\right)$ is larger than $d$, so that
the functions\newline
$\{g(\mu_l^{\rm S1},\cdot)\}_{1\leq l\leq d}$ are linearly independent (otherwise,
$(\delta^{\rm S1}_k g)(\mu^{\rm S1}_{k+1},x^{\rm S1}_{k+1})=0$ for some $k$ in Algorithm~\ref{algo0}).
\end{itemize}

\begin{algorithm}[h!]
	\caption{Offline stage ${\rm EIM}^{\rm S1}$}
	\label{algo0}
	\begin{algorithmic}[1]
        \STATE {Choose $d>1$}
        \hfill \COMMENT{Number of interpolation points}
	\STATE {Set $k:=1$}
        \STATE {Compute $\displaystyle \mu^{\rm S1}_1:=\underset{\mu\in \mathcal{P}}{\textnormal{argmax}}\|g(\mu,\cdot)\|_{\Omega}$}
        \STATE {Compute $\displaystyle x^{\rm S1}_1:=\underset{x\in\Omega}{\textnormal{argmax}}|g(\mu^{\rm S1}_1,x)|$}
        \hfill \COMMENT{First interpolation point}
	\STATE {Set $\displaystyle q^{\rm S1}_1(\cdot):=\frac{g(\mu^{\rm S1}_1,\cdot)}{g(\mu^{\rm S1}_1,x^{\rm S1}_1)}$}
       \hfill \COMMENT{First basis function}
        \STATE {Set $B^{\rm S1}_{1,1}:=1$}
       \hfill \COMMENT{Initialize matrix $B^{\rm S1}$}
        \WHILE {$k < d$} 
		\STATE Compute $\displaystyle \mu^{\rm S1}_{k+1}:=\underset{\mu\in \mathcal{P}}{\textnormal{argmax}}\|(\delta^{\rm S1}_k g)(\mu,\cdot)\|_{\Omega}$
                \STATE Compute $\displaystyle x^{\rm S1}_{k+1}:=\underset{x\in\Omega}{\textnormal{argmax}}|(\delta^{\rm S1}_k g)(\mu^{\rm S1}_{k+1},x)|$
                \hfill \COMMENT{$(k+1)$-th interpolation point}  
                \STATE Set $\displaystyle q^{\rm S1}_{k+1}(\cdot):=\frac{(\delta^{\rm S1}_k g)(\mu^{\rm S1}_{k+1},\cdot)}{(\delta^{\rm S1}_k g)(\mu^{\rm S1}_{k+1},x^{\rm S1}_{k+1})}$
                \hfill \COMMENT{$(k+1)$-th basis function}
                \STATE Set $\displaystyle B^{\rm S1}_{k+1,i}:=q^{\rm S1}_{i}(x^{\rm S1}_{k+1})$, for all $1\leq i\leq {k+1}$
                \hfill \COMMENT{Increment matrix $B^{\rm S1}$}
                \STATE $k\leftarrow k+1$
                \hfill \COMMENT{Increment the size of the decomposition}
	\ENDWHILE
\end{algorithmic}
\end{algorithm}

The online stage of ${\rm EIM}^{\rm S1}$ amounts to~\eqref{eq:onlinea}-\eqref{eq:onlineapb} for $k=d$.
This yields
\begin{equation}
\label{eq:onlinea1}
\left(I^{\rm S1}_{d} g\right)(\mu,x) := \sum_{m=1}^{d} \lambda^{\rm S1}_m(\mu) q^{\rm S1}_{m}(x),
\end{equation}
where the functions $\lambda_m^{\rm S1}(\mu)$, $1\leq m\leq d$, solve the linear system
\begin{equation}
\label{eq:onlinea1pb}
\sum_{m=1}^{d}B^{\rm S1}_{l,m}{\lambda}^{\rm S1}_m(\mu)=g(\mu,x^{\rm S1}_l), \qquad \forall 1\leq l\leq d.
\end{equation}
Eliminating $\lambda_m^{\rm S1}(\mu)$ leads to
\begin{equation}
\label{eq:approx1}
\left(I^{\rm S1}_{d} g\right)(\mu,x) = \sum_{m=1}^{d} \sum_{l=1}^{d} (B^{\rm S1})^{-1}_{m,l} g(\mu,x^{\rm S1}_l) q^{\rm S1}_{m}(x),
\end{equation}
where the matrix $\left(B^{\rm S1}\right)^{-1}$ is computed during the offline stage.

The function $I^{\rm S1}_{d} g$ can be rewritten without using the functions $\{q^{\rm S1}_{m}\}_{1\leq m\leq d}$.
By construction, it is clear that 
$\underset{1\leq m\leq d}{\rm Span}\left(q^{\rm S1}_{m}(\cdot)\right)=\underset{1\leq m\leq d}{\rm Span}\left(g(\mu^{\rm S1}_m,\cdot)\right)$.
Therefore, there exists a matrix $\Gamma^{\rm S1}\in\mathbb{R}^{d\times d}$ such that, for all $1\leq l\leq d$,
\begin{equation}
\label{eq:def:Gammaz}
\sum_{m=1}^d(\Gamma^{\rm S1})_{l,m}q^{\rm S1}_{m}(x)= g(\mu^{\rm S1}_{l},x),\qquad \forall x\in\Omega.
\end{equation}
Owing to assumption (H), the matrix $\Gamma^{\rm S1}$ is invertible.
The construction of the matrix $\Gamma^{\rm S1}$ is detailed in Lemma~\ref{propinterpbetwpts}.
Using~\eqref{eq:def:Gammaz} in~\eqref{eq:approx1} yields
\begin{equation}
\label{eq:approx2}
\begin{aligned}
\left(I^{\rm S1}_{d} g\right)(\mu,x) &= \sum_{m=1}^{d} \sum_{l=1}^{d} \sum_{r=1}^{d} (B^{\rm S1})^{-1}_{m,l} (\Gamma^{\rm S1})^{-1}_{m,r} g(\mu,x^{\rm S1}_l)g(\mu^{\rm S1}_{r},x)\\
&= \sum_{l=1}^{d} \sum_{r=1}^{d} \Delta^{\rm S1}_{l,r} g(\mu,x^{\rm S1}_l)g(\mu^{\rm S1}_{r},x),
\end{aligned}
\end{equation}
where the matrix $\Delta^{\rm S1}:=(\Gamma^{\rm S1} (B^{\rm S1})^t)^{-1}$ can be computed during the offline stage.

\begin{lem}
\label{propinterpbetwpts}
The matrix $\Gamma^{\rm S1}$ can be constructed recursively in the loop in $k$ of~Algorithm~\ref{algo0} in the following way:
\begin{itemize}
 \item $k=1$:
\begin{equation*}
(\Gamma^{\rm S1})_{1,1}=g(\mu_{1}^{\rm S1}, x_{1}^{\rm S1}), 
\end{equation*}
 \item $k\rightarrow k+1$: 
\begin{equation*}
\begin{alignedat}{3}
(\Gamma^{\rm S1})_{k+1,k+1}&=(\delta_{k}^{\rm S1} g)(\mu^{\rm S1}_{k+1}, x^{\rm S1}_{k+1}),&&\\ 
(\Gamma^{\rm S1})_{l,k+1}&=0,&\qquad \forall 1\leq l\leq k,&\\
(\Gamma^{\rm S1})_{k+1,l}&=\kappa^{\rm S1}_{l},&\qquad \forall 1\leq l\leq k,&\\ 
\end{alignedat}
\end{equation*}
where the vector $\kappa^{\rm S1}$ is such that $\sum_{m=1}^k (B^{\rm S1})_{l,m}\kappa^{\rm S1}_{m}=g(\mu^{\rm S1}_{k+1}, x^{\rm S1}_{l})$, for all $1\leq l\leq k$.
\end{itemize}
\end{lem}
\begin{proof}
The case $k=1$ results from line 5 of Algorithm~\ref{algo0}. Suppose that the assertion holds at rank $k$.
Using the definition~\eqref{eq:def:Gammaz} of $\Gamma^{\rm S1}$ at rank $(k+1)$, for all $1\leq l\leq k$ and all $x\in\Omega$,
we infer that 
\begin{equation*}
(\Gamma^{\rm S1})_{l,k+1}q^{\rm S1}_{k+1}(x)+\sum_{m=1}^{k}(\Gamma^{\rm S1})_{l,m}q^{\rm S1}_{m}(x)=g(\mu^{\rm S1}_{l},x).
\end{equation*}
Using the same definition at rank $k$ leads to $(\Gamma^{\rm S1})_{l,k+1}=0$ for all $1\leq l\leq k$.
Then, using the same definition for $l=k+1$, we infer that
\begin{equation*}
(\Gamma^{\rm S1})_{k+1,k+1}q^{\rm S1}_{k+1}(x)+\sum_{m=1}^{k}(\Gamma^{\rm S1})_{k+1,m}q^{\rm S1}_{m}(x)=g(\mu^{\rm S1}_{k+1},x). 
\end{equation*}
Using line 10 of Algorithm~\ref{algo0}, we identify
$(\Gamma^{\rm S1})_{k+1,k+1}=(\delta^{\rm S1}_{k} g)(\mu^{\rm S1}_{k+1}, x^{\rm S1}_{k+1})$ and
$\sum_{m=1}^{k}(\Gamma^{\rm S1})_{k+1,m}q^{\rm S1}_{m}(x)=(I^{\rm S1}_k g)(\mu^{\rm S1}_{k+1}, x)$.
From~\eqref{eq:onlinea1}-\eqref{eq:onlinea1pb}, we infer that
\begin{equation*}
\sum_{m=1}^{k}(\Gamma^{\rm S1})_{k+1,m}q^{\rm S1}_{m}(x)=\sum_{l=1}^k\sum_{m=1}^k (B^{\rm S1})^{-1}_{m,l}q^{\rm S1}_{m}(x)g(\mu^{\rm S1}_{k+1}, x^{\rm S1}_{l}).
\end{equation*}
Therefore, $(\Gamma^{\rm S1})_{k+1,m}=\sum_{l=1}^k(B^{\rm S1})^{-1}_{m,l}g(\mu^{\rm S1}_{k+1}, x^{\rm S1}_{l})$, finishing the proof.\qed
\end{proof}

We recall the interpolation property of $I^{\rm S1}_d g$; see~\cite[Lemma~1]{Maday}:
\begin{prop}[Interpolation property]
\label{interp1}
For all $1\leq m\leq d$,
\begin{equation*}
\left\{
\begin{aligned}
(I^{\rm S1}_d g)(\mu,x^{\rm S1}_m) &= g(\mu,x^{\rm S1}_m), \quad \textnormal{for all } \mu\in\mathcal{P},\\
(I^{\rm S1}_d g)(\mu^{\rm S1}_m,x) &= g(\mu^{\rm S1}_m,x), \quad \textnormal{for all } x\in\Omega.
\end{aligned}
\right.
\end{equation*}
\end{prop}
\begin{proof}
Let $\mu\in\mathcal{P}$.
Since $B^{\rm S1}$ is invertible, $(I^{\rm S1}_d g)(\mu,\cdot)$ is uniquely determined by~\eqref{eq:onlinea1}-\eqref{eq:onlinea1pb}
as an element of $\underset{1\leq m\leq d}{\rm Span}\left(q^{\rm S1}_m(\cdot)\right)=
\underset{1\leq m\leq d}{\rm Span}\left(g(\mu^{\rm S1}_m,\cdot)\right)$.
Replacing the values of the coefficients of $B^{\rm S1}$
defined in line 11 of Algorithm~1 in~\eqref{eq:onlinea1pb}, we infer that 
\begin{equation*}
(I^{\rm S1}_d g)(\mu,x^{\rm S1}_l)=
 \sum_{m=1}^{d} \lambda^{\rm S1}_m(\mu) q^{\rm S1}_{m}(x^{\rm S1}_l)=
\sum_{m=1}^{d} B^{\rm S1}_{l,m}\lambda^{\rm S1}_m(\mu)=g(\mu,x^{\rm S1}_l),
\end{equation*}
for all $1\leq l\leq d$. Therefore, for all $1\leq m\leq d$, $(I^{\rm S1}_d g)(\mu^{\rm S1}_l,\cdot)$ is the
element of $\underset{1\leq m\leq d}{\rm Span}\left(g(\mu^{\rm S1}_m,\cdot)\right)$ such that $(I^{\rm S1}_d g)(\mu^{\rm S1}_m,x^{\rm S1}_m)=
g(\mu^{\rm S1}_m,x^{\rm S1}_m)$. The linear independence of $\{g(\mu^{\rm S1}_m,\cdot)\}_{1\leq m\leq d}$ yields
$(I^{\rm S1}_d g)(\mu^{\rm S1}_m,x) = g(\mu^{\rm S1}_m,x)$, for all $1\leq m\leq d$ and all $x\in\Omega$.
\end{proof}

\begin{rem}[Alternative expression for $I^{\rm S1}_d$]
It is readily verified that
\begin{equation}
\label{eq:onlinea2}
\left(I^{\rm S1}_{d} g\right)(\mu,x) := \sum_{m=1}^{d} \hat{\lambda}^{\rm S1}_m(x) g(\mu,x^{\rm S1}_m),
\end{equation}
where the functions $\hat{\lambda}_m^{\rm S1}(x)$, $1\leq m\leq d$, solve the linear system
\begin{equation}
\label{eq:onlinea2pb}
\sum_{m=1}^{d}(B^{\rm S1})^t_{l,m}{\hat{\lambda}}^{\rm S1}_m(x)=q^{\rm S1}_{l}(x), \qquad \forall 1\leq l\leq d.
\end{equation}
\end{rem}

\begin{rem}[Stabilized EIM]
When $d$ is large, it can be interesting to stabilize each step in the $k$-th loop of the
offline stage of the EIM with respect to round-off errors, in the same spirit
as the stabilized Gram-Schmidt procedure; see~\cite{M2AN}. The numerical simulations presented in Section~\ref{sec:manyqueries}
use this stabilized version.
\end{rem}

\subsection{{\em Slice 2}}

A variant of Algorithm~\ref{algo0} is obtained by switching the roles of $\mu$ and $x$ in the offline stage.
We denote this variant by ${\rm EIM}^{\rm S2}$, S2 refering to {\em Slice 2}.
Fix an integer $d>1$ (the total number of interpolation points). Then, for all $1\leq k\leq d$, the rank-$k$ approximation operator
$I^{\rm S2}_{k}$ is defined as
\begin{equation}
\label{eq:onlineaa}
\left(I^{\rm S2}_{k} g\right)(\mu,x) := \sum_{m=1}^{k} \lambda^{\rm S2}_m(x) q^{\rm S2}_{m}(\mu),
\end{equation}
where the functions ${\lambda}^{\rm S2}_m(x)$, $1\leq m\leq k$, solve the linear system
\begin{equation}
\label{eq:onlineaapb}
\sum_{m=1}^{k}B^{\rm S2}_{l,m}{\lambda}^{\rm S2}_m(x)=g(\mu^{\rm S2}_l, x), \qquad \forall 1\leq l\leq k.
\end{equation}
The functions $q^{\rm S2}_{m}(\cdot)$ and the matrices $B^{\rm S2}\in\mathbb{R}^{k\times k}$, which are
lower triangular with unity diagonal, are constructed as described in Algorithm~\ref{algo1},
where $\delta^{\rm S2}_k={\rm Id}-I^{\rm S2}_{k}$ and $\|\cdot\|_{\mathcal{P}}$ is a norm
on $\mathcal{P}$, for instance the $L^\infty\left(\mathcal{P}\right)$- or the $L^2\left(\mathcal{P}\right)$-norm.
Note that Algorithm~\ref{algo1} also constructs the set of points $\{\mu_l^{\rm S2}\}_{1\leq l\leq d}$ in $\mathcal{P}$ used in~\eqref{eq:onlineaapb},
and a set of points $\{x_l^{\rm S2}\}_{1\leq l\leq d}$ in $\Omega$.
Similarly to (H), we assume that the dimension of
$\underset{1\leq l\leq d}{\rm Span}\left(g(\cdot,x_l^{\rm S2})\right)$ is $d$.
\begin{algorithm}[h!]
	\caption{Offline stage ${\rm EIM}^{\rm S2}$}
	\label{algo1}
	\begin{algorithmic}[1]
        \STATE {Choose $d>1$}
        \hfill \COMMENT{Number of interpolation points}
	\STATE {Set $k:=1$}
        \STATE {Compute $\displaystyle x^{\rm S2}_1:=\underset{x\in\Omega}{\textnormal{argmax}}\|g(\cdot,x)\|_{\mathcal{P}}$}
        \STATE {Compute $\displaystyle \mu^{\rm S2}_1:=\underset{\mu\in \mathcal{P}}{\textnormal{argmax}}|g(\mu,x^{\rm S2}_1)|$}
        \hfill \COMMENT{First interpolation point}
	\STATE {Set $\displaystyle q^{\rm S2}_1(\cdot):=\frac{g(\cdot,x^{\rm S2}_1)}{g(\mu^{\rm S2}_1,x^{\rm S2}_1)}$}
       \hfill \COMMENT{First basis function}
        \STATE {Set $B^{\rm S2}_{1,1}:=1$}
       \hfill \COMMENT{Initialize matrix $B^{\rm S2}$}
        \WHILE {$k < d$} 
                \STATE Compute $\displaystyle x^{\rm S2}_{k+1}:=\underset{x\in\Omega}{\textnormal{argmax}}\|(\delta^{\rm S2}_k g)(\cdot,x)\|_{\mathcal{P}}$
                \STATE Compute $\displaystyle \mu^{\rm S2}_{k+1}:=\underset{\mu\in \mathcal{P}}{\textnormal{argmax}}|(\delta^{\rm S2}_k g)(\mu,x^{\rm S2}_{k+1})|$
                \hfill \COMMENT{$(k+1)$-th interpolation point}  
                \STATE Set $\displaystyle q^{\rm S2}_{k+1}(\cdot):=\frac{(\delta^{\rm S2}_k g)(\cdot,x^{\rm S2}_{k+1})}{(\delta^{\rm S2}_k g)(\mu^{\rm S2}_{k+1},x^{\rm S2}_{k+1})}$
                \hfill \COMMENT{$(k+1)$-th basis function}
                \STATE Set $\displaystyle B^{\rm S2}_{k+1,i}:=q^{\rm S2}_i(\mu^{\rm S2}_{k+1})$, for all $1\leq i\leq {k+1}$
                \hfill \COMMENT{Increment matrix $B^{\rm S2}$}
                \STATE $k\leftarrow k+1$
                \hfill \COMMENT{Increment the size of the decomposition}
	\ENDWHILE
\end{algorithmic}
\end{algorithm}

In the same fashion as in Section~\ref{sec:slice1}, the approximation of $g$ is given by
\begin{equation}
\label{eq:approx1slice2}
\left(I^{\rm S2}_{d} g\right)(\mu,x) = \sum_{m=1}^{d} \sum_{l=1}^{d} (B^{\rm S2})^{-1}_{m,l} g(\mu^{\rm S2}_l,x) q^{\rm S2}_{m}(\mu),
\end{equation}
where the matrix $\left(B^{\rm S2}\right)^{-1}$ is computed during the offline stage.

\begin{rem}[Equivalence between S1 and S2]
\label{equivS1S2}
Since the roles of $x$ and $\mu$ are not symmetric, the algorithms ${\rm EIM}^{\rm S1}$ and ${\rm EIM}^{\rm S2}$ lead in general
to different approximations of the function $g$. However,
in the case where the norms are $\|\cdot\|_{\Omega}=\|\cdot\|_{L^{\infty}(\Omega)}$ and $\|\cdot\|_{\mathcal{P}}=\|\cdot\|_{L^{\infty}(\mathcal{P})}$,
it can be shown by induction on $k$ that the same sets of points $\mu_l$ and $x_l$ are selected by Algorithms~\ref{algo0}
and~\ref{algo1}, and that the same matrices $B$ and $\Gamma$ are computed. Therefore,
$(I_d^{\rm S1}g)(\mu,x)=(I_d^{\rm S2}g)(\mu,x)$ for all $(\mu,x)\in\mathcal{P}\times\Omega$.
\end{rem}

There exists a matrix $\Gamma^{\rm S2}\in\mathbb{R}^{d\times d}$ such that, for all $1\leq l\leq d$,
\begin{equation}
\label{eq:compq_m}
\sum_{m=1}^d(\Gamma^{\rm S2})_{l,m}q^{\rm S2}_{m}(\mu)= g(\mu,x^{\rm S2}_{l}),\qquad \forall \mu\in\mathcal{P}.
\end{equation}
The construction of the matrix $\Gamma^{\rm S2}$ is detailed in Lemma~\ref{propinterpbetwpts2}.
Using~\eqref{eq:compq_m} in~\eqref{eq:approx1slice2} yields
\begin{equation}
\label{eq:approx2slice2}
\begin{aligned}
\left(I^{\rm S2}_{d} g\right)(\mu,x) = \sum_{l=1}^{d} \sum_{r=1}^{d} \Delta^{\rm S2}_{l,r} g(\mu^{\rm S2}_l,x) g(\mu,x^{\rm S2}_{r}).
\end{aligned}
\end{equation}
where the matrix $\Delta^{\rm S2}:=(\Gamma^{\rm S2} (B^{\rm S2})^t)^{-1}$ can be computed during the offline stage.

\begin{lem}
\label{propinterpbetwpts2}
The matrix $\Gamma^{\rm S2}$ can be constructed recursively in the loop in $k$ of Algorithm~\ref{algo1} in the following way:
\begin{itemize}
 \item $k=1$:
\begin{equation*}
(\Gamma^{\rm S2})_{1,1}=g(\mu^{\rm S2}_{1}, x^{\rm S2}_{1}), 
\end{equation*}
 \item $k\rightarrow k+1$: 
\begin{equation*}
\begin{alignedat}{3}
(\Gamma^{\rm S2})_{k+1,k+1}&=(\delta^{\rm S2}_{k} g)(\mu^{\rm S2}_{k+1}, x^{\rm S2}_{k+1}),&&\\ 
(\Gamma^{\rm S2})_{l,k+1}&=0,&\qquad \forall 1\leq l\leq k,&\\
(\Gamma^{\rm S2})_{k+1,l}&=\kappa^{\rm S2}_{l},&\qquad \forall 1\leq l\leq k,&\\ 
\end{alignedat}
\end{equation*}
where the vector $\kappa^{\rm S2}$ is such that $\sum_{m=1}^k (B^{\rm S2})_{l,m}\kappa^{\rm S2}_{m}=g(\mu^{\rm S2}_{l}, x^{\rm S2}_{k+1})$, for all $1\leq l\leq k$.
\end{itemize}
\end{lem}
\begin{proof}
Similar to that of Lemma~\ref{propinterpbetwpts}.
\end{proof}

The following interpolation property holds.
\begin{prop}[Interpolation property]
\label{interp2}
For all $1\leq m\leq d$,
\begin{equation*}
\left\{
\begin{aligned}
(I^{\rm S2}_d g)(\mu,x^{\rm S2}_m) &=  g(\mu,x^{\rm S2}_m), \quad \textnormal{for all } \mu\in\mathcal{P},\\
(I^{\rm S2}_d g)(\mu^{\rm S2}_m,x) &=  g(\mu^{\rm S2}_m,x), \quad \textnormal{for all } x\in\Omega.
\end{aligned}
\right.
\end{equation*}
\end{prop}
\begin{proof}
Similar to that of Proposition~\ref{interp1}.
\end{proof}

\section{Nonintrusive procedure}
\label{sec:nonint}

The goal of this section is to obtain a nonintrusive approximation, using an offline-online procedure, of the following quantities:
\begin{equation}
\label{eq:targetqqty}
Q_t(\mu)=\sum_{s=1}^{\varsigma}\int_{\Omega}g_s(\mu,x)\Psi_{s,t}(x) dx, \qquad \forall t\in\{1\ldots N\},
\end{equation}
where $\varsigma\geq 2$, while $N$ is supposed to be large.
The functions $\Psi_{s,t}$ are basis functions or products of basis functions involved in the evaluation of the entries of the vector 
$C_\mu$ and the matrix $A_\mu$ in~\eqref{eq:fullpbmat}, see Section~\ref{sec:manyqueries} for various examples.
We want the procedure to be robust with respect to $N$. This means that EIM algorithms can only be carried out to approximate
the functions $(\mu,x)\mapsto g_s(\mu,x)$ and not the functions $(\mu,x)\mapsto g_s(\mu,x)\Psi_{s,t}(x)$.
An example is 
\begin{equation*}
a_\mu(u,v)=\int_{\Omega}g(\mu,x){\nabla} u(x)\cdot{\nabla} v(x) dx
\end{equation*}
so that
\begin{equation*}
(A_\mu)_{i,j}=\int_{\Omega}g(\mu,x){\nabla} \varphi_i(x)\cdot{\nabla} \varphi_j(x) dx,
\end{equation*}
which corresponds to~\eqref{eq:targetqqty}
with $\varsigma=1$, $t=(i,j)$, and $\Psi_{1,t}(x)={\nabla}\varphi_i(x)\cdot{\nabla}\varphi_j(x)$.

The main results of this section are approximations of~\eqref{eq:targetqqty} in the form
\begin{equation}
\label{eq:targetqqty2}
Q_t(\mu)\approx\sum_{r=1}^{d^z}\beta_r(\mu)Q_t(\mu_r),
\end{equation}
for some integer $d^z$, coefficients $\{\beta_r(\mu)\}_{1\leq r\leq d^z}$ and parameter values $\{\mu_r\}_{1\leq r\leq d^z}$.

\subsection{Affine dependence available}
\label{sec:affdepavailable}

To illustrate our main idea, we first consider the case where the function $g_s$ only depends on $\mu$, but not on $x$. This
corresponds to the case where an affine dependence is already available, so that
\begin{equation}
\label{eq:targetqqtyaff}
Q_t(\mu) = \sum_{s=1}^{\varsigma}g_s(\mu)\int_{\Omega}\Psi_{s,t}(x) dx, \qquad \forall t\in\{1\ldots N\}.
\end{equation}
The key idea is now to apply an EIM procedure to the function $g_s(\mu)$ seen as a two-variable function
\begin{equation*}
\gamma:(\mu,s)\mapsto g_s(\mu),
\end{equation*}
where $\mu\in\mathcal{P}$ and $1\leq s\leq \varsigma$.
The two approximation procedures S1($\gamma$) and S2($\gamma$) are possible for the approximation of $\gamma(\mu,s)$,
where now $s$ plays the role that $x$ played in Section~\ref{sec:synthEIM}, and where we indicate specifically in the notation that
these procedures are related to the approximation of $\gamma(\mu,s)$.
The finite sets used in practice to compute the argmax appearing in the offline stage of
the approximation procedures S1($\gamma$) and S2($\gamma$) are $\mathcal{P}_{\rm trial}$ and $\{1\ldots\varsigma\}$.
We keep the same notation as before for the constructed matrices $B$, the vector-valued functions $q_m(\cdot)$,
and the selected points $\mu_m$, while we introduce the indices $s_l$
selected by the EIM procedures to approximate $\gamma(\mu,s)$.
Employing for instance the procedure S1 and using~\eqref{eq:approx2} leads to
\begin{equation*}
g_s(\mu)\approx (I_d^{{\rm S1}(\gamma)}\gamma)(\mu,s)=\sum_{r=1}^d\underbrace{\left\{\sum_{l=1}^d \Delta^{{\rm S1}(\gamma)}_{l,r}
g_{s^{{\rm S1}(\gamma)}_l}(\mu)\right\}}_{:=\beta_r(\mu)}g_{s}(\mu^{{\rm S1}(\gamma)}_r),
\end{equation*}
where $d\leq\varsigma$ is the number of points used in the EIM applied to $\gamma$.
Using this approximation in~\eqref{eq:targetqqtyaff} and exchanging the order of summations leads to
\begin{equation*}
\begin{aligned}
Q_t(\mu) &\approx \sum_{s=1}^{\varsigma}\sum_{r=1}^{d}\beta_r(\mu)g_{s}(\mu^{{\rm S1}(\gamma)}_r)\int_{\Omega}\Psi_{s,t}(x) dx\\
&=\sum_{r=1}^{d}\beta_r(\mu)Q_t(\mu^{{\rm S1}(\gamma)}_r),
\end{aligned}
\end{equation*}
which corresponds to~\eqref{eq:targetqqty2}. A similar nonintrusive approximation can be derived using the procedure S2$(\gamma)$;
details are skipped for brevity.

\subsection{Nonaffine dependence}
\label{sec:nonaffdep}

When the affine dependence considered in Section~\ref{sec:affdepavailable} is not available, the first classical step consists in
approximating the functions $g_s(\mu,x)$ for all $1\leq s\leq\varsigma$ using the procedure S1 or S2.
This leads to the construction of $\varsigma$ sets of points $x$, points $\mu$, matrices $B$, $\Gamma$, $\Delta$, and vector-valued functions $q(\cdot)$.
We denote these quantities with an additional index $s$; for instance, ${\rm EIM}^{\rm S1}$ carried out on $g_s(\mu,x)$ leads to the construction of the
vector-valued functions $q^{\rm S1}_{s}(\cdot)$, of components $q^{\rm S1}_{s,m}:x\mapsto q^{\rm S1}_{s,m}(x)$, for all
$1\leq m\leq d$. For simplicity and without loss of generality, we assume that each EIM algorithm stops at the same rank $d$. 

Consider the procedure S1. 
Injecting the approximation~\eqref{eq:approx2} of $g_s(\mu,x)$, for all $1\leq s\leq \varsigma$,
into~\eqref{eq:targetqqty2} yields an approximation of $Q_t(\mu)$, which we denote by $(\mathcal{I}^{\rm S1}_d Q_t)(\mu)$
and which is given by
\begin{equation}
\label{approxqqtS1O1}
\begin{aligned}
(\mathcal{I}^{\rm S1}_d Q_t)(\mu)&:= \sum_{s=1}^{\varsigma}\int_{\Omega}({I}^{\rm S1}_d g_s)(\mu,x)\Psi_{s,t}(x)dx\\
&=\sum_{s=1}^{\varsigma} \sum_{m=1}^{d} \sum_{l=1}^{d} (\Delta^{\rm S1}_s)_{l,m} g_s(\mu,x^{\rm S1}_{s,l})\int_{\Omega}g_s(\mu^{\rm S1}_{s,m},x)\Psi_{s,t}(x)dx.
\end{aligned}
\end{equation}
The key idea is that~\eqref{approxqqtS1O1} is a linear form in a vector $z\in\mathbb{R}^{\varsigma d}$, whose
components, denoted by $z_p(\mu)$, $1\leq p\leq \varsigma d$ (the index $p$ collects the indices $s,m$ in~\eqref{approxqqtS1O1}), contain all the $\mu$-dependencies:
\begin{equation}
\label{eq:IS1O1z}
(\mathcal{I}_d^{\rm S1}Q_t)(\mu) = \sum_{p=1}^{\varsigma d} z_p(\mu)\mathcal{Q}_{t,p},
\end{equation}
where
\begin{equation*}
z_p(\mu):=\left\{
\begin{alignedat}{3}
&\sum_{l=1}^d (\Delta^{\rm S1}_1)_{l,m}g_1(\mu,x^{\rm S1}_{1,l}),&\qquad &1\leq m\leq d,\quad  p=m,\\
&\sum_{l=1}^d (\Delta^{\rm S1}_2)_{l,m}g_2(\mu,x^{\rm S1}_{2,l}),&\qquad &1\leq m\leq d,\quad  p=m+d,\\
&&\vdots&&&\\
&\sum_{l=1}^d (\Delta^{\rm S1}_\varsigma)_{l,m}g_\varsigma(\mu,x^{\rm S1}_{\varsigma,l}),&\qquad &1\leq m\leq d,\quad  p=m+(\varsigma-1)d,
\end{alignedat}\right.
\end{equation*}
\begin{equation*}
\mathcal{Q}_{t,p}:=\left\{
\begin{alignedat}{3}
&\int_{\Omega}g_1(\mu^{\rm S1}_{1,m},x)\Psi_{1,t}(x)dx,&\qquad &1\leq m\leq d,\quad  p=m,\\
&\int_{\Omega}g_2(\mu^{\rm S1}_{2,m},x)\Psi_{2,t}(x)dx,&\qquad &1\leq m\leq d,\quad  p=m+d,\\
&&\vdots&&&\\
&\int_{\Omega}g_\varsigma(\mu^{\rm S1}_{\varsigma,m},x)\Psi_{\varsigma,t}(x)dx,&\qquad &1\leq m\leq d,\quad  p=m+(\varsigma-1)d.
\end{alignedat}\right.
\end{equation*}
Now, following the same procedure as in Section~\ref{sec:affdepavailable}, a nonintrusive approximation for $Q_t(\mu)$ of the
form~\eqref{eq:targetqqty2} is achieved by applying another EIM to $z_p(\mu)$ seen as the two-variable function
\begin{equation*}
\zeta:(\mu,p)\mapsto z_p(\mu),
\end{equation*}
where $\mu\in\mathcal{P}$ and $1\leq p\leq \varsigma d$.
The finite sets used in practice to compute the argmax appearing in the offline stage of
the approximation procedures S1($\zeta$) and S2($\zeta$) are $\mathcal{P}_{\rm trial}$ and $\{1\ldots\varsigma d\}$.
We denote $d^z\leq \varsigma d$ the number of points used in this second EIM.

Injecting the approximation of $\zeta(\mu,p)$ using S1($\zeta$) into the right-hand side of~\eqref{eq:IS1O1z} yields
\begin{equation}
\label{2approxqqtS1O1}
\begin{aligned}
(\mathcal{I}^{\rm S1}_d Q_t)(\mu)&\approx \sum_{p=1}^{\varsigma d} (I^{{\rm S1}(\zeta)}_d \zeta)(\mu,p)\mathcal{Q}_{t,p}\\
&= \sum_{p=1}^{\varsigma d} \sum_{l=1}^{d^z}\sum_{r=1}^{d^z}
\Delta^{{\rm S1}(\zeta)}_{l,r}z_{p^{{\rm S1}(\zeta)}_l}(\mu)
z_p(\mu_r^{{\rm S1}(\zeta)})\mathcal{Q}_{t,p}.
\end{aligned}
\end{equation}
Switching the order of summations in~\eqref{2approxqqtS1O1} leads to
\begin{equation*}
\begin{aligned}
(\mathcal{I}^{\rm S1}_d Q_t)(\mu)&\approx \sum_{r=1}^{d^z} \sum_{l=1}^{d^z}
\Delta^{{\rm S1}(\zeta)}_{l,r}z_{p^{{\rm S1}(\zeta)}_l}(\mu)
\sum_{p=1}^{\varsigma d}z_p(\mu_r^{{\rm S1}(\zeta)})\mathcal{Q}_{t,p}\\
&=\sum_{r=1}^{d^z} \sum_{l=1}^{d^z}
\Delta^{{\rm S1}(\zeta)}_{l,r}z_{p^{{\rm S1}(\zeta)}_l}(\mu)
(\mathcal{I}^{\rm S1}_d Q_t)(\mu_r^{{\rm S1}(\zeta)}),
\end{aligned}
\end{equation*}
where~\eqref{eq:IS1O1z} has been used in the second line. Replacing $\mathcal{I}^{\rm S1}_d Q_t$ by $Q_t$ yields the
following nonintrusive approximation formula for $Q_t(\mu)$:
\begin{equation}
\label{eq:nonintform1}
Q_t(\mu)\approx \sum_{r=1}^{d^z} \underbrace{\left\{ \sum_{l=1}^{d^z}
\Delta^{{\rm S1}(\zeta)}_{l,r}z_{p^{{\rm S1}(\zeta)}_l}(\mu)\right\}}_{:=\beta_r(\mu)}Q_t(\mu_r^{{\rm S1}(\zeta)}).
\end{equation}
In the same fashion, injecting the approximation of $\zeta(\mu,p)$ using S2($\zeta$) in the right-hand side
of~\eqref{eq:IS1O1z} yields the following nonintrusive approximation formula of $Q_t(\mu)$:
\begin{equation}
\label{eq:nonintform2}
Q_t(\mu)\approx \sum_{r=1}^{d^z} \left\{ \sum_{l=1}^{d^z}
\Delta^{{\rm S2}(\zeta)}_{l,r}z_{p^{{\rm S2}(\zeta)}_l}(\mu)\right\}Q_t(\mu_r^{{\rm S2}(\zeta)}).
\end{equation}

It is also possible to use the procedure S2 to approximate the functions $g_s(\mu,x)$, leading to the construction of another vector $z_p(\mu)$,
which can be approximated using either S1($\zeta$) or S2($\zeta$). Details are skipped for brevity.

\begin{rem}[Alternative procedure]
Expression~\eqref{eq:IS1O1z} also holds with the following choices for $z_p(\mu)$ and $\mathcal{Q}_{t,p}$:
\label{remotherdecomp}
\begin{equation}
\label{choicezpmu}
z_p(\mu):=\left\{
\begin{alignedat}{3}
&\sum_{l=1}^d (B^{\rm S1}_1)_{m,l}^{-1}g_1(\mu,x^{\rm S1}_{1,l}),&\qquad &1\leq m\leq d,\quad  p=m,\\
&\sum_{l=1}^d (B^{\rm S1}_2)_{m,l}^{-1}g_2(\mu,x^{\rm S1}_{2,l}),&\qquad &1\leq m\leq d,\quad  p=m+d,\\
&&\vdots&&&\\
&\sum_{l=1}^d (B^{\rm S1}_\varsigma)_{m,l}^{-1}g_\varsigma(\mu,x^{\rm S1}_{\varsigma,l}),&\qquad &1\leq m\leq d,\quad  p=m+(\varsigma-1)d,
\end{alignedat}\right.
\end{equation}
\begin{equation}
\label{choiceQ}
\mathcal{Q}_{t,p}:=\left\{
\begin{alignedat}{3}
&\int_{\Omega}q^{\rm S1}_{1,m}(x)\Psi_{1,t}(x)dx,&\qquad &1\leq m\leq d,\quad  p=m,\\
&\int_{\Omega}q^{\rm S1}_{2,m}(x)\Psi_{2,t}(x)dx,&\qquad &1\leq m\leq d,\quad  p=m+d,\\
&&\vdots&&&\\
&\int_{\Omega}q^{\rm S1}_{\varsigma,m}(x)\Psi_{\varsigma,t}(x)dx,&\qquad &1\leq m\leq d,\quad  p=m+(\varsigma-1)d,
\end{alignedat}\right.
\end{equation}
leading to the same kind of nonintrusive procedures.
\end{rem}

\begin{rem}[Computational cost]
Using the approximation formula~\eqref{eq:nonintform1} to gain nonintrusiveness leads to additional computations mainly 
in the offline stage, corresponding to the EIM applied to $z_p(\mu)$.
In the online stage, the classical formula~\eqref{approxqqtS1O1} (which is intrusive) contains $\varsigma d^2$ terms,
whereas~\eqref{eq:nonintform1} contains $d_z^2$ terms. Both online formulae are of complexity independent of $N$,
and the difference of computational cost between them depends on the values of $\varsigma d^2$ and $d_z^2$.
\end{rem}

\section{Nonintrusive RBM for aeroacoustic problems}
\label{sec:manyqueries}

In this section, we consider discrete variational formulations
of aeroacoustic problems modeled by the Helmholtz equation or the convected Helmholtz equation.
The finite element method (FEM) and the boundary element method (BEM) are used to obtain the matrix and the right-hand side of the problem
\cite{casenavephd,jcp}.
The entries of both quantities are of the form $Q_t(\mu)$ as defined in~\eqref{eq:targetqqty}. For the matrix, the index~$t$ in $\Psi_{s,t}(x)$
refers to the product of two finite element basis functions,
while for the right-hand side, the index $t$ refers to the basis functions themselves.

In our simulations, we use the nonintrusive formula~\eqref{eq:nonintform2} for the matrix and the right-hand side,
with $L^\infty(\Omega)$- and $L^\infty(\mathcal{P})$-norms (so that S1 and S2 are equivalent, see Remark~\ref{equivS1S2}),
and the choice~\eqref{choicezpmu} for $z_p(\mu)$ and~\eqref{choiceQ} for $\mathcal{Q}_{t,p}$.
We only need to compute matrix-vector products
involving $A_{\mu}$ and scalar products to precompute in the offline stage all the quantities needed to construct efficiently
the reduced problem and compute the error bound in the online stage.

In Section~\ref{sec:stabilityconst}, we discuss some issues concerning the computation of the inf-sup constant associated with the
discrete problem; recall that an approximation of this constant is needed to evaluate the a posteriori error bound in the RBM.
Then, we present nonintrusive RBM simulations for three aeroacoustic
problems in Sections~\ref{sec:fin:1}, \ref{sec:fin:2}, and \ref{scalableRB}.
The in-house EADS software ACTIPOLE has been used in our simulations.

\subsection{Computation of the inf-sup constant}
\label{sec:stabilityconst}

Applying the Successive Constraint Method (SCM, see~\cite{Huynh}) as an online-efficient
procedure for computing the inf-sup constant requires to solve constrained linear optimization problems with a number of constraints proportional
to the square of the number of selected parameter values.
This is particularly demanding when considering reduced basis strategies for the (convected) Helmholtz equation approximated by BEM
with the frequency as a parameter, 
since it is required to take a rather large value of $d^z$ to obtain an accurate approximation in the form of the affine decomposition~\eqref{eq:nonintform2}.
Alternatively, the power iteration method (see~\cite{poweriteration}) associated with the inverse matrix can be used to 
approximately compute
the smallest eigenvalue of the eigenvalue problems to be solved when evaluating the inf-sup constant.
This would imply solving many eigenvalue problems associated with the inverse operator, and therefore does not appear to be
reasonable for the present industrial test cases.

We proceed as follows in our test cases.
In Sections~\ref{sec:fin:1} and~\ref{sec:fin:2}, we compute a single value of the inf-sup constant (for centered values of the parameters) and use it for any error bound evaluation.
Even if the inf-sup constant depends on the parameters, its values are not expected to exhibit significant variations since the
considered formulations do not feature any resonant frequency (for Helmholtz problems with resonant frequencies,
see~\cite{Rozzahelm}).
In Section~\ref{scalableRB}, the test case has much more unknowns than those from the two previous sections.
Therefore, we do not compute the inf-sup constant, so as to temper the offline computational cost.
Dealing further with the derivation of an online-efficient strategy to compute the inf-sup constant for test cases with a large
number of unknowns goes beyond the present scope.
Finally, for simplicity, we take the Euclidian norm of the discrete vectors in the computation of the a posteriori error bound in the
RBM.

\subsection{An optimization problem for an impedant object in the air at rest}
\label{sec:fin:1}

\begin{figure}[htbp]
	\centering
	\includegraphics [width=0.45\textwidth] {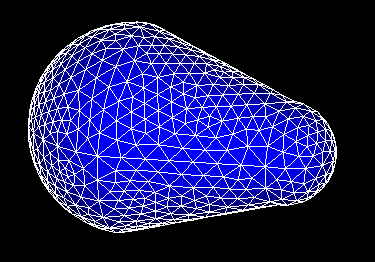}
	\includegraphics [width=0.45\textwidth] {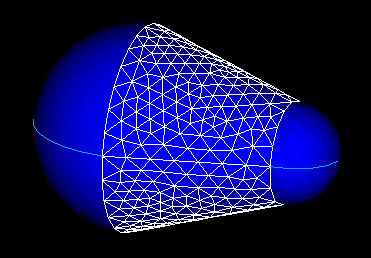}
	\caption{Test case 1. Left: mesh for test case 1. Right: impedant surface $\Gamma_2$}
\label{fig:many:cone}
\end{figure}

Consider the object whose mesh is represented in the left panel of Figure~\ref{fig:many:cone}. The surface of this object, denoted by $\Gamma$,
is partitioned into three simply connected disjoint zones denoted by $\Gamma_1$, $\Gamma_2$, and $\Gamma_3$ respectively. The surface $\Gamma_2$ is represented in
the right panel of Figure~\ref{fig:many:cone}. On each of these zones,
a Robin boundary condition is enforced with a specific impedance coefficient $\mu_i$ for $i\in\{1,2,3\}$.
Thus, the impedance coefficient on $\Gamma$, denoted by $\mu_{\Gamma}$, is piecewise constant and takes the form
$\mu_{\Gamma}(x)=\mu_1 1\!\!1_{\Gamma_1}(x)+\mu_2 1\!\!1_{\Gamma_2}(x)+\mu_3 1\!\!1_{\Gamma_3}(x)$, for all $x\in\Gamma$,
where $1\!\!1_{\Gamma_i}$, $i\in\{1,2,3\}$, are characteristic functions.
The source is a plane wave whose wave vector is supported by the axis of symmetry of the object, creating an incident acoustic
pressure field denoted by $p^{\text{inc}}_{\mu_0}$, where $\mu_0=\frac{\omega}{c}$ is the wavenumber of the source with $\omega$ the pulsation of the source
and $c$ the speed of sound in the air at rest. The variational formulation of the problem is as follows:
Find $\left(\chi,\lambda\right)\in H^\frac{1}{2}(\Gamma)\times L^2(\Gamma)$
such that for all $(\hat{\chi},\hat{\lambda})\in H^\frac{1}{2}(\Gamma)\times L^2(\Gamma)$,
\begin{equation}
\label{eq:varf}
 \left\{
\begin{aligned}
\left(N_{\mu_0}\chi-\frac{i{\mu_0}}{2\mu_{\Gamma}}\chi, \hat{\chi}\right)_{\Gamma} + \left(\tilde{D}_{\mu_0}\lambda, \hat{\chi}\right)_{\Gamma}&= \left(\gamma_1 p^{\text{inc}}_{\mu_0}, \hat{\chi}\right)_{\Gamma},\\
\left(\hat{\lambda},D_{\mu_0} \chi\right)_{\Gamma} - \left(\hat{\lambda},S_{\mu_0}\lambda +\frac{i\mu_{\Gamma}}{2{\mu_0}}\lambda\right)_{\Gamma} &= -\left(\hat{\lambda},\gamma_0 p^{\text{inc}}_{\mu_0}\right)_{\Gamma},
\end{aligned}
\right.
\end{equation}
where $\left(\cdot,\cdot\right)_{\Gamma}$ denotes the extension of the $L^2(\Gamma)$-inner product to the duality pairing on
$H^{-\frac{1}{2}}(\Gamma)\times H^{\frac{1}{2}}(\Gamma)$ and $\gamma_0$ and $\gamma_1$ respectively denote the Dirichlet and Neumann
traces on $\Gamma$.
The operators $N_{\mu_0}$, $D_{\mu_0}$, $\tilde{D}_{\mu_0}$, and $S_{\mu_0}$ are boundary integral operators expressed in terms of
the Green kernel $G_{\mu_0}(x,y)=\frac{\exp(i{\mu_0}|x-y|)}{4\pi|x-y|}$ associated with the Helmholtz equation at wavenumber ${\mu_0}$.
The pressure field around the object is then obtained by applying a representation formula to $\left(\chi,\lambda\right)$, the solution to~\eqref{eq:varf}.
We refer to~\cite[Chapter~2]{casenavephd} for more details on the formulation~\eqref{eq:varf} and its well-posedness.
The considered finite-dimensional approximation of~\eqref{eq:varf} has $2240$ unknowns.

The parameters of the problem are the frequency of the source and the impedance coefficient of each of the three zones composing the surface of the object.
The frequency varies from $487$ to $1082$ Hz, and each impedance coefficient varies from $1$ to $5$.
The quantity of interest is the far-field acoustic pressure along the axis of symmetry of the object, but in the opposite direction of the source.
A goal-oriented RBM is carried out to select a basis of $\hat{n}=20$ truth solutions using the nonintrusive formula~\eqref{eq:nonintform2} to approximate the matrix,
the right-hand side of the direct problem, and the right-hand side of the adjoint problem needed to evaluate the quantity of interest.
For the matrix, the approximation procedure S1 is applied to
\begin{equation}
\label{eq:many:matimp1}
g({\mu_0},r):=\exp\left(i{\mu_0} r\right),~r=\left|x-y\right|,~x,y\in\Gamma,
\end{equation}
and the procedure S2($\zeta$) is applied to
\begin{equation}
\label{eq:many:matimp2}
z_p({\mu_0},\mu_1, \mu_2, \mu_3):=\left\{
\begin{alignedat}{2}
&\lambda^{\rm S1}_m({\mu_0}),&\qquad &1\leq m\leq d, \quad p=m,\\
&{\mu_0}\lambda^{\rm S1}_m({\mu_0}),&\qquad &1\leq m\leq d, \quad p=m+d,\\
&\mu_0^2\lambda^{\rm S1}_m({\mu_0}),&\qquad &1\leq m\leq d, \quad p=m+2d,\\
&\frac{{\mu_0}}{\mu_1},&\qquad &p=3d+1,\\
&\frac{\mu_1}{{\mu_0}},&\qquad &p=3d+2,\\
&\frac{{\mu_0}}{\mu_2},&\qquad &p=3d+3,\\
&\frac{\mu_2}{{\mu_0}},&\qquad &p=3d+4,\\
&\frac{{\mu_0}}{\mu_3},&\qquad &p=3d+5,\\
&\frac{\mu_3}{{\mu_0}},&\qquad &p=3d+6,
\end{alignedat}\right.
\end{equation}
where we recall that $\lambda^{\rm S1}_m({\mu_0})=\sum_{l=1}^d (B^{\rm S1})_{m,l}^{-1}g(\mu_0,x^{\rm S1}_{l})$.
For the approximation formula of the right-hand side of the direct and dual problems, the procedure S1 is applied to
\begin{equation}
\label{eq:many:vecmono10}
g({\mu_0},x):=\exp\left(i{\mu_0}\vec{d}\cdot\vec{x}\right),~x\in\Gamma,
\end{equation}
where $\vec{d}$ is respectively the direction of the incoming plane wave and the direction of measure of the far-field; and the procedure S2($\zeta$) is applied to
\begin{equation}
\label{eq:many:vecmono20}
z_m({\mu_0}):=
\begin{alignedat}{2}
&\lambda^{\rm S1}_m({\mu_0}),&\qquad &1\leq m\leq d, \quad p=m.
\end{alignedat}
\end{equation}
The EIM algorithms are carried out with $d=13$ and $d^z=20$ for the matrix, and $d=13$ and $d^z=13$ for the right-hand side of the direct and dual problems.
Over the considered parameter values, the relative error for the three nonintrusive formulae is of the order of $10^{-12}$ (in Frobenius norm for the matrix and Euclidian norm for the vectors).
The maximum error bound (over a discretization $\mathcal{P}_{\rm trial}$) is of the order of $10^{-6}$, the online stage takes $2.8\times 10^{-3}$ s to compute a reduced solution and the error bound,
while the full direct problem is solved in about $30$ s in parallel on $4$ processors, which corresponds to an acceleration factor of $10^4$.

Let us now illustrate the interest of the RBM on an optimization problem, which is a natural context 
where the parametrized problem has to be solved for many values of the parameters.
Consider a set of values $\mu_{0_i}$, $1\leq i\leq \ell$, for the wavenumber
of the source and denote by $J_i(\mu_1, \mu_2, \mu_3)$, the quantity of interest computed for the wavenumber $\mu_{0_i}$ of the source and depending on the three impedance coefficients.
Consider the following cost function:
\begin{equation}
\label{eq:cost_function}
(\mu_1, \mu_2, \mu_3)\mapsto \mathcal{J}(\mu_1, \mu_2, \mu_3):=\sum_{i=1}^\ell \alpha_i J_i(\mu_1, \mu_2, \mu_3) + h(\mu_1, \mu_2, \mu_3).
\end{equation}
The goal of the study is to find values of the impedance coefficients that minimize the cost function~\eqref{eq:cost_function}.
With such a cost function, we can minimize the far-field acoustic pressure scattered by the object, taking into account that some frequencies are more harmful than others
for the human ear (through the weights $\alpha_i$), and that some treatments of the object surface to modify the impedance
coefficients are more expensive than others (through the function $h$).
To illustrate the procedure, we choose $\ell=20$, $\alpha_i=2$ for $1\leq i \leq 7$, $\alpha_i=1$ for $8\leq i \leq 13$, and $\alpha_i=3$ for $14\leq i \leq 20$, and
\begin{equation*}
h(\mu_1, \mu_2, \mu_3)=\frac{1}{6}(0.2\mu_1^{-0.5}+0.3\mu_2^{-0.8}+0.5\mu_3^{-1})-8. 
\end{equation*}
The cost function is computed for $1000$ values of the impedance coefficients (each coefficient being sampled by $10$ values). Notice that for each evaluation of the cost function,
we need to compute the solution of the aeroacoustic problem for $20$ values of the frequency.
Using the online stage, the minimum of the cost function over this sample of impedance coefficients is $0.366$, reached for $(\mu_1, \mu_2, \mu_3)=(2.8, 1, 1.9)$, and is found in less than $24$ s.

Figure~\ref{fig:many:javacone} shows a screenshot of a java applet computing the quantity of interest at $50$ values of the frequency, and at values of the impedance coefficients
selected by the user.

\begin{figure}[htbp]
	\centering
	\includegraphics [width=0.7\textwidth] {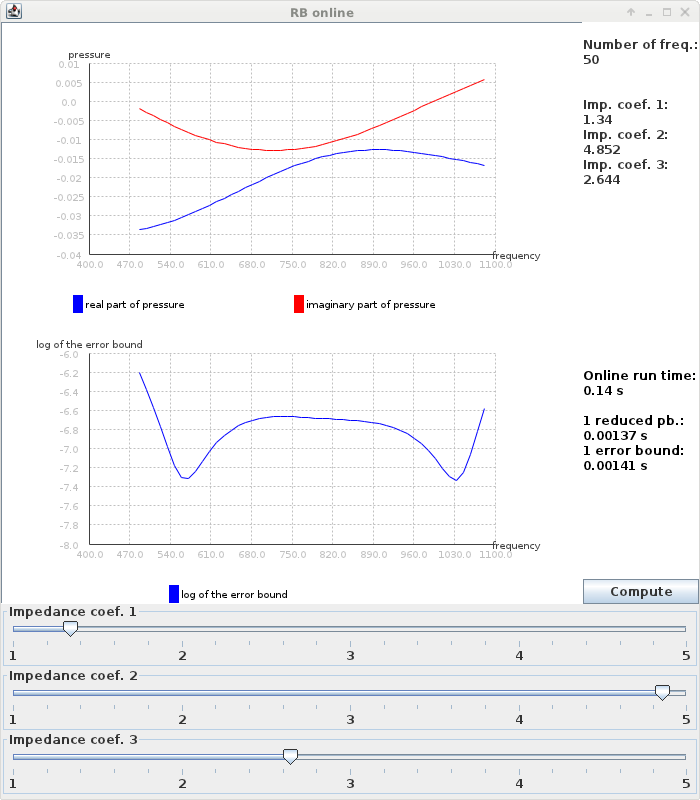}
	\caption{Java applet for the online stage of the RBM for test case 1. Top panel: real part and imaginary part of the
far-field pressure for $50$ values of the frequency. Middle panel: error bound as a function of frequency.
Bottom panel: selection of the impedance coefficients}
\label{fig:many:javacone}
\end{figure}

\subsection{An uncertainty quantification problem for an object surrounded by a potential flow}
\label{sec:fin:2}

Consider an ellipsoid with major axis directed along the $z$-axis. This object is included inside a larger ball,
see Figure~\ref{fig:mesh1}.
The external border of the ball after discretization is denoted by $\Gamma_\infty$.
The complement of the ellipsoid in the ball is denoted by $\Omega^-$.
A potential flow is precomputed around the ellipsoid and inside the ball,
such that the flow is uniform outside the ball, of Mach number $0.3$, and directed along the $z$-axis.
An acoustic monopole source lies upstream of the object, on the $z$-axis as well.

\begin{figure}[h!]
 \centering
\includegraphics[width=0.47\textwidth]{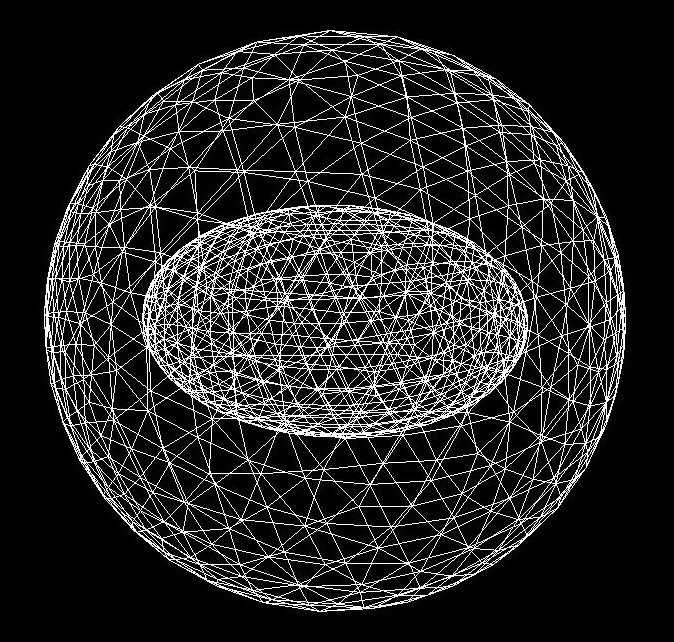}
\includegraphics[width=0.5\textwidth]{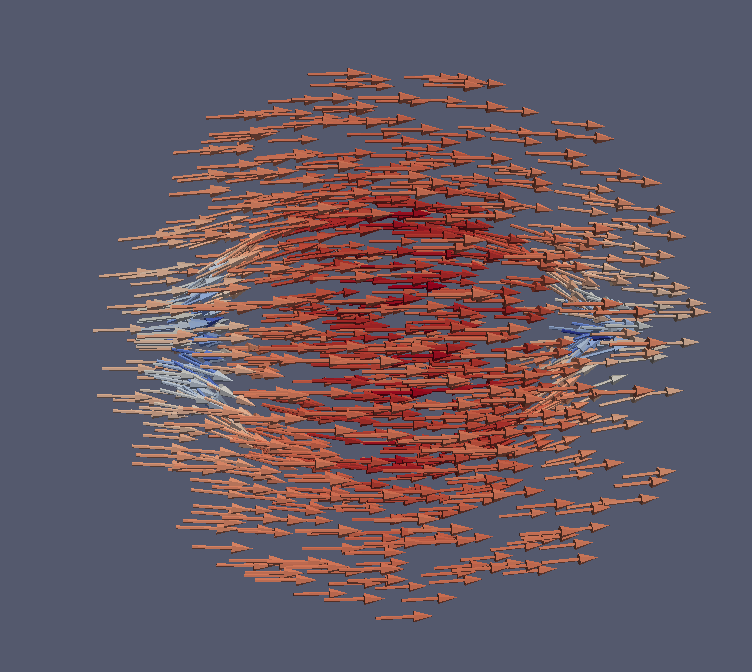}
 \caption{Test case 2. Left: representation of the mesh. Right: potential flow around the ellipsoid}
\label{fig:mesh1}
\end{figure}

The considered formulation is a coupled BEM-FEM formulation, see~\cite[Chapter~3]{casenavephd} and~\cite{jcp} for more
details and well-posedness.
It consists in (i) applying a change of variables to transform the convected Helmholtz equation into the classical Helmholtz equation
outside the ball, in order to apply a standard BEM on $\Gamma_{\infty}$, and (ii) stabilizing the formulation
to avoid resonant frequencies associated with the eigenvalues of the Laplacian inside the ball of boundary $\Gamma_\infty$.
Consider the product space $\mathbb{H}:=H^{1}\left({\Omega^-}\right)\times H^{-\frac{1}{2}}\left({\Gamma_\infty}\right)\times H^{1}({\Gamma_\infty})$
with inner product
\begin{equation*}
\left(\left(\Phi,\lambda,p\right),\left(\Phi^t,\lambda^t,p^t\right)\right)_{\mathbb{H}}:=\left(\Phi, \Phi^t\right)_{H^{1}\left({\Omega^-}\right)}+\left(\lambda,\lambda^t\right)_{H^{-\frac{1}{2}}\left({\Gamma_\infty}\right)}
+\left(p,p^t\right)_{H^{1}\left({\Gamma_\infty}\right)}.
\end{equation*}
The weak formulation is:
Find $\left(\Phi,\lambda,p\right)\in \mathbb{H}$ such that $\forall\left(\Phi^t,\lambda^t, p^t\right)\in \mathbb{H}$,
\begin{subequations}
\label{eq:weakcouptrans2}
\begin{align}
&\mathcal{V}_{\mu_0}(\Phi,\Phi^t)+\left(N_{\mu_0}({\gamma_0^-}\Phi),{\gamma_0^-}\Phi^t\right)_{\Gamma_{\infty}}\!\!\!\!+\left(\left({\tilde{D}_{\mu_0}}
\!-\!\tfrac{1}{2}I\right)(\lambda),{\gamma_0^-}\Phi^t\right)_{\Gamma_{\infty}}\!\!\!\!\nonumber
\\&\qquad\qquad\qquad\qquad\qquad\qquad\qquad\qquad\qquad\qquad\qquad~
= \left({\gamma_1}{f^{\rm inc}_{\mu_0}},{\gamma_0^-}\Phi^t\right)_{\Gamma_{\infty}}\!\!,\label{eq:weakcouptrans21}\\
&\left(\lambda^t,\left({D}_{\mu_0}\!-\!\tfrac{1}{2}I\right)({\gamma_0^-}\Phi)\right)_{\Gamma_{\infty}}\!\!\!\!\!-\left(\lambda^t,S_{\mu_0}(\lambda)\right)_{\Gamma_{\infty}}
\!\!\!\!-i\left(\lambda^t,p\right)_{\Gamma_{\infty}}\!\!\!\! = -\left(\lambda^t,{\gamma_0}{f^{\rm inc}_{\mu_0}}\right)_{\Gamma_{\infty}}\!\!,\label{eq:weakcouptrans22}\\
&\left(N_{\mu_0}({\gamma_0^-}\Phi),p^t\right)_{\Gamma_{\infty}}\!\!\!\!\!+\left(\left({\tilde{D}}_{\mu_0}\!+\!\tfrac{1}{2}I\right)(\lambda),p^t\right)_{\Gamma_{\infty}}
\!\!\!\!-\delta_{\Gamma_\infty}(p,p^t)= \left({\gamma_1}{f^{\rm inc}_{\mu_0}},p^t\right)_{\Gamma_\infty}\!\!,\label{eq:weakcouptrans23}
\end{align}
\end{subequations}
where $\left(\cdot,\cdot\right)_{\Gamma_\infty}$ denotes the extension of the $L^2(\Gamma_\infty)$-inner product to the duality
pairing on $H^{-\frac{1}{2}}\left(\Gamma_\infty\right)\times H^{\frac{1}{2}}\left(\Gamma_\infty\right)$,
$\gamma_0^-$ is the interior Dirichlet trace on $\Gamma_\infty$, and $f^{\rm inc}_{\mu_0}$ is related to the source term with ${\mu_0}$ the wavenumber of the source (so that the frequency of the source is $\frac{{\mu_0}c}{2\pi}$ in the air at rest),
and where
\begin{equation*}
\delta_{\Gamma_\infty}(p,q):=\left(\vec{\nabla}_{\Gamma_\infty}p, \vec{\nabla}_{\Gamma_\infty}q\right)_{\Gamma_\infty}+
\left(p,q\right)_{\Gamma_\infty},
\end{equation*}
with $\vec{\nabla}_{\Gamma_\infty}$ the surfacic gradient on $\Gamma_\infty$, and
\begin{equation*}
\mathcal{V}_{\mu_0}(\Phi,\Phi^t):=\int_{{\Omega^-}}\Xi\vec{\nabla}\overline{\Phi}\cdot\vec{\nabla}{{\Phi}^t}-\mu_0^2\int_{{\Omega^-}}
\beta \overline{\Phi}{{\Phi}^t}+i{\mu_0} \int_{{\Omega^-}}\vec{V}\cdot\left(\overline{\Phi}\vec{\nabla}{{\Phi}^t}-{{\Phi}^t}
\vec{\nabla}\overline{\Phi}\right),
\end{equation*}
where 
$\beta := r\left(\left(\varsigma+\gamma_\infty^2 P\right)^2-\gamma_\infty^4 M_\infty^2\right)$,
$\vec{V} := r\left(\left(\varsigma+\gamma_\infty^2 P\right)\mathcal{N}\vec{M} -\gamma_\infty^3\vec{M}_\infty\right)$,
$\Xi := r\mathcal{N}\mathcal{O}\mathcal{N}$ with
$r := \frac{\rho}{\rho_\infty}$,
$\varsigma :=\frac{c_\infty}{c}$,
$\gamma_\infty := \frac{1}{\sqrt{1-M_\infty^2}}$,
$P := \vec{M}\cdot\vec{M}_\infty$,
$\mathcal{N} := I+C_\infty\vec{M}_\infty\vec{M}^T_\infty$,
$\mathcal{O} := I-\vec{M}\vec{M}^T$,
and $C_\infty := \frac{\gamma_\infty-1}{M_{\infty}^2}$.
In the above notation, the subscript $\infty$ is used for quantities outside the ball,
$\rho$ is the density of the flow, $c$ is the speed of sound when the flow is at rest, and
$\vec{M}=\frac{\vec{v}}{c}$, where $\vec{v}$ is the velocity of the flow.
The considered finite-dimensional approximation of~\eqref{eq:weakcouptrans2} has $1711$ unknowns.

The potential flow, represented in the right panel of Figure~\ref{fig:mesh1}, is part of the data of the problem.
We perturb this flow uniformly in space.
Although the boundary condition on the solid surface $\Gamma$ and the transmission condition on $\Gamma_{\infty}$ are violated by a nonzero flow perturbation,
the present study can be viewed as a first step towards quantifying uncertainties on the potential flow and
their impact on a quantity of interest.
The flow perturbation takes the form $\delta \vec{M} = \mu_1 \vec{e}_x + \mu_2 \vec{e}_y + \mu_3 \vec{e}_z$.
The quantity of interest is the acoustic pressure at a point located on the axis of symmetry, downstream of the object.
The parameters of the problem are the frequency of the source, and the magnitude of the uniform perturbations of the potential flow in each
Cartesian direction. The frequency varies from $487$ to $1082$ Hz, and the magnitude of the uniform perturbations of the flow varies from $0$ to $0.1$.
A goal-oriented RBM is carried out to select a basis of $\hat{n}=20$ truth solutions using the nonintrusive formula~\eqref{eq:nonintform2} to approximate the matrix of the problem,
the right-hand side of the direct problem,
and the right-hand side of the adjoint problem corresponding to our quantity of interest.
For the matrix, the approximation procedure S1 is applied to
\begin{equation*}
g({\mu_0},r):=\exp\left(i{\mu_0} r\right),~r=\left|x-y\right|,~x,y\in\Gamma_\infty,
\end{equation*}
and the procedure S2($\zeta$) is applied to
\begin{equation*}
z_p({\mu_0},\mu_1, \mu_2, \mu_3):=\left\{
\begin{alignedat}{5}
&\lambda^{\rm S1}_m({\mu_0}),&\qquad &1\leq m\leq d, & \quad &p=m,&\\
&{\mu_0}\lambda^{\rm S1}_m({\mu_0}),&\qquad &1\leq m\leq d, &\quad  &p=m+d,&\\
&\mu_0^2\lambda^{\rm S1}_m({\mu_0}),&\qquad &1\leq m\leq d, &\quad & p=m+2d,&\\
&1,&\qquad &&\quad& p=3d+1,&\\
&{\mu_0},&\qquad &&\quad& p=3d+2,&\\
&\mu_0^2,&\qquad &&\quad &p=3d+3,&\\
&\mu_0^2\mu_3,&\qquad &&\quad& p=3d+4,&\\
&\mu_0^2\mu_3^2,&\qquad &&\quad& p=3d+5,&\\
&{\mu_0}\mu_i,&\qquad &1\leq i\leq 3, &\quad& p=3d+5+i,&\\
&{\mu_0}\mu_i\mu_3,&\qquad &1\leq i\leq 3, &\quad& p=3d+8+i,&\\
&\mu_i\mu_j,&\qquad &1\leq i,j\leq 3, &\quad& p=3d+11+i+3(j-1),&
\end{alignedat}\right.
\end{equation*}
where these parameter dependencies have been identified upon injecting $\vec{M}\rightarrow \vec{M} + \delta \vec{M}$ in~\eqref{eq:weakcouptrans2},
while using that $\vec{M}_{\infty}$ is collinear to $\vec{e}_z$.
For the right-hand side of the direct and dual problems, the approximation procedure S1 is applied to
\begin{equation*}
g({\mu_0},x):=\exp\left(i{\mu_0} |x-x_0|\right),~x\in\Gamma_\infty,
\end{equation*}
where $x_0$ is respectively the position of the source and the point where the quantity of interest is computed, and the
approximation procedure S2($\zeta$) is applied to
\begin{equation*}
z_p({\mu_0}):=\left\{
\begin{alignedat}{2}
&\lambda^{\rm S1}_m({\mu_0}),&\qquad &1\leq m\leq d, \quad p=m\\
&{\mu_0}\lambda^{\rm S1}_m({\mu_0}),&\qquad &1\leq m\leq d, \quad p=m+d.
\end{alignedat}\right.
\end{equation*}
The EIM algorithms are carried out with $d=13$ and $d^z=25$ for the matrix, and $d=13$ and $d^z=18$ for the right-hand side of the direct and dual problems.
Over the considered parameter values, the relative error for the three nonintrusive formulae is of the order of $10^{-12}$ (in Frobenius norm for the matrix and Euclidian norm for the vectors).
The maximum error bound (over a discretization $\mathcal{P}_{\rm trial}$) is of the order of $10^{-7}$, the online stage takes $2.8\times 10^{-3}$ s to compute a reduced solution and the error bound,
while the full direct problem is solved in about $14$~s, which corresponds to an acceleration factor of $5\times 10^3$.

To illustrate the procedure, we suppose that the perturbation of the potential flow is modelled by random variables: the law of $\mu_1$ is a truncated Gaussian,
that of $\mu_2$ is a uniform law, and that of $\mu_3$ is a truncated log-normal law.
The goal is to compute the probability density function of the quantity of interest.
Figure~\ref{fig:many:javaec} shows a screenshot of a java applet computing an histogram of the values taken by the quantity of interest, at a frequency selected by the user.

\begin{figure}[htbp]
	\centering
	\includegraphics [width=0.8\textwidth] {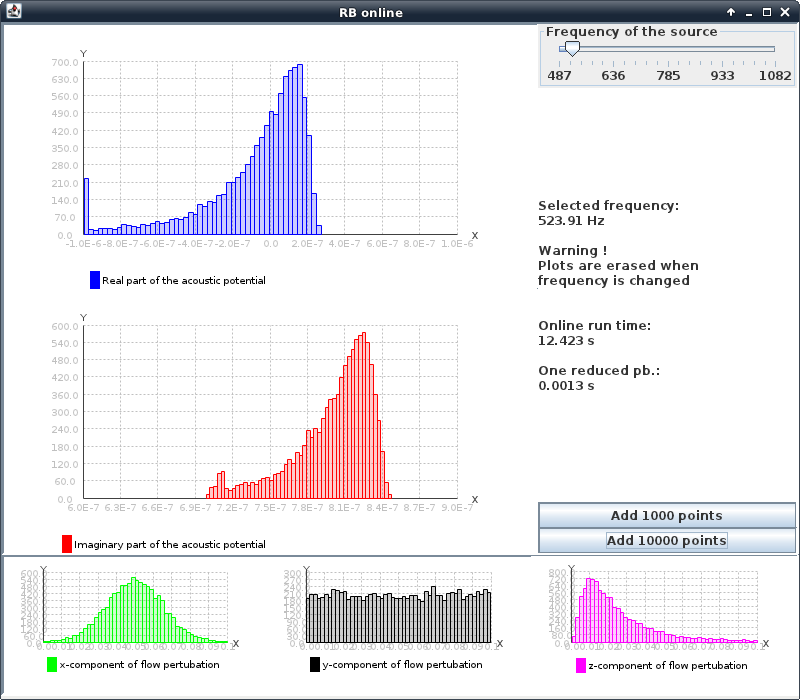}
	\caption{Java applet for the online stage of the RBM for test case 2.
Top panel: histograms of the real part and imaginary part of the quantity of interest. Bottom panel: 
histograms of the three components of the perturbation of the flow}
\label{fig:many:javaec}
\end{figure}

\subsection{A scalable RBM implementation applied to an industrial test case of an impedant aircraft in the air at rest}
\label{scalableRB}

In BEM implementations for the Helmholtz equation, the Fast Multipole Method (FMM) allows one to approximately
compute matrix-vector products, and then approximately solve linear systems
using iterative methods, in complexity scaling with $n\log n$, where $n$ denotes the number of unknowns \cite{precomp,sylvand}.
For boundary integral systems, the matrices are dense, and have a priori $n^2$ nonzero complex coefficients.
In this section, we consider a test case where the matrices $A_{\mu^{{\rm S2}(\zeta)}_m}$, where the $\mu^{{\rm S2}(\zeta)}_m$ are parameter values selected when applying the 
nonintrusive formula~\eqref{eq:nonintform2} to the approximation of $A_\mu$, are so large that they
cannot be stored on the hard drive of the computer used for the simulations. Therefore, each time a matrix-vector product is carried out,
the matrix is assembled, and the FMM is used.

We consider the same problem as in Section~\ref{sec:fin:1}, i.e., the scattering of an incoming acoustic field by an object
whose surface has been coated on three zones by three impedant materials. However,
the considered scattering object is now an aircraft, see Figure~\ref{fig:many:aircraft}.
\begin{figure}[htbp]
	\centering
	\includegraphics [width=0.8\textwidth] {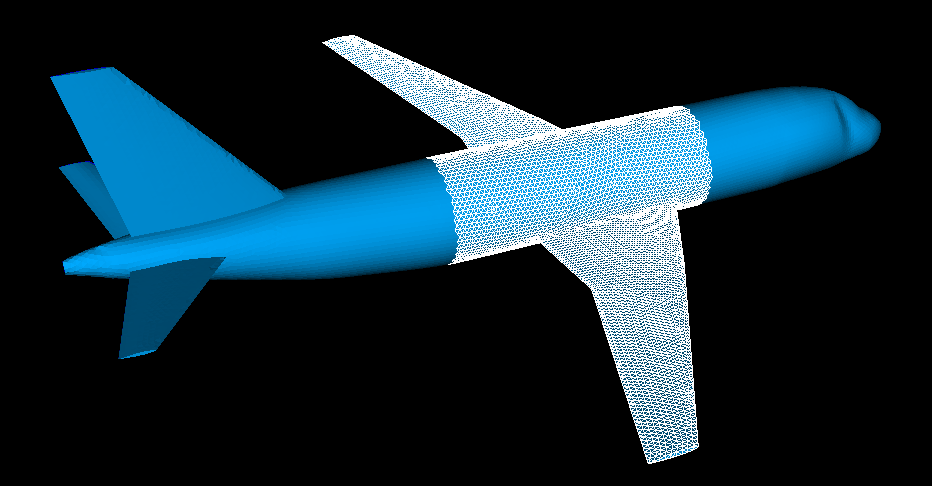}
	\caption{Second impedant surface, with the finest mesh for test case 3}
\label{fig:many:aircraft}
\end{figure}
Two meshes are considered: the coarser one leading to a discrete formulation with $11831$ unknowns, the finer one leading to a discrete formulation with $60866$ unknowns.
The source is an acoustic monopole, located under the right wing of the plane.
The parameters of the problem are the frequency of the source, and the impedance of the three zones composing the surface of the aircraft.
The frequency varies from $27$ to $135$ Hz, and each impedance coefficient varies from $1$ to $2$. We take $532400$ parameter values in
$\mathcal{P}_{\rm trial}$ ($400$ values for the frequency and $11$ values for each impedance coefficient).

First, the RBM is applied to the problem on the coarser mesh.
To recover the affine dependence assumption, we use the nonintrusive approximation formula~\eqref{eq:nonintform2}, with~\eqref{eq:many:matimp1}-\eqref{eq:many:matimp2} for the matrix decomposition
(with now $d=35$ and $d^z=50$) and~\eqref{eq:many:vecmono10}-\eqref{eq:many:vecmono20} for the decomposition of the right-hand side of the problem (with $d=50$ and $d^z=60$).
With $\hat{n}=30$ basis vectors selected by the greedy algorithm, the relative error between the direct solution and the reduced solution, in Euclidian norm,
at the value of the parameters that maximizes the error bound, is less than $3\%$.
The two steps of the procedure with highest computational complexity are the matrix-vector products in FMM and the 
exploration of $\mathcal{P}_{\rm trial}$ by the greedy algorithm. The former can be parallelized, and this is so in ACTIPOLE,
while the latter is trivial to parallelize. Therefore, the procedure is expected to be extremely efficient on distributed
architectures, that is, to be scalable with respect to the number of processors.

We now consider the finer mesh. Each time a vector $U_{\mu_j}$ is added to the reduced basis, we have to compute the $d^z=50$ matrix-vector
products $A_{\mu_m^{{\rm S2}(\zeta)}}U_{\mu_j}$, $1\leq m\leq d^z$, where the $\mu_m^{{\rm S2}(\zeta)}$
are the values of the parameter in the nonintrusive approximation formula~\eqref{eq:nonintform2}. Therefore, in addition to the resolution of the direct problem,
$50$ matrices have to be assembled at each step of the greedy algorithm, which is time-consuming. However, once a matrix is constructed, it is relatively cheap to compute many matrix-vector products
with the same matrix. Hence, a greedy algorithm is not considered on the finer mesh, but
the values of the parameters selected by the greedy algorithm on the coarser mesh are directly used to build the reduced basis on the
finer mesh. This way, the $50$ matrices are constructed once, and only
$30$ matrix-vector products (corresponding to $\hat{n}=30$ values of the parameter selected by the greedy algorithm on
the coarser mesh) are carried out for each matrix.
The simulations have been performed on a laptop with a quadricore CPU, and $4$ GB of RAM.
The formula~\eqref{eq:nonintform2} allows us to directly use the FMM. Without the FMM, this simulation on this computer would have been impossible, since one matrix needs
$60$~GB to be stored.
An approach attempting to compute and store the $50$ matrices of the decomposition would need $3$~TB of memory.

The online stage takes $1.5\times 10^{-2}$ s to compute a reduced solution and the error bound,
while the full direct problem is solved in about $40$ minutes, which corresponds to an acceleration factor of $1.6\times 10^5$.
The offline stages are computed in about 2 days, and the last step of the greedy algorithm in the offline stage of the RBM with the coarser mesh
takes 1 hour. The FMM we used computes matrix-vector products with a relative accuracy of approximately $10^{-3}$; therefore, we cannot expect to achieve a much more accurate RB approximation.

The acoustic field in the exterior domain is computed from the solution to~\eqref{eq:varf} using a representation formula, which is a linear operation.
We consider the acoustic field on an array of $1681$ points located behind the aircraft. 
We can precompute this field using the vectors of the reduced basis as solutions, and the quantity of interest is directly obtained at any parameter value
from these precomputed fields and the components of the reduced solutions.
Figure~\ref{fig:many:javascalable} shows a screenshot of a java applet computing this acoustic field at a set of parameters selected by the user
(frequency and impedance coefficients).
Consider the following parameter values: frequency = $122.3$ Hz, $\mu_1=1.21$, $\mu_2=1.87$, and $\mu_3=1.45$.
The error bound is $5.4\times 10^{-4}$, and the relative error between the direct solution and the reduced solution, in Euclidian norm, is $1\%$.
On the array of $1681$ points located behind the aircraft, the relative error for the scattered acoustic field is $1.4\%$. Figure~\ref{fig:many:diffRB} shows the corresponding acoustic pressure fields
and the difference between the reduced basis and direct solutions.

\begin{figure}[htbp]
	\centering
	\includegraphics [width=0.7\textwidth] {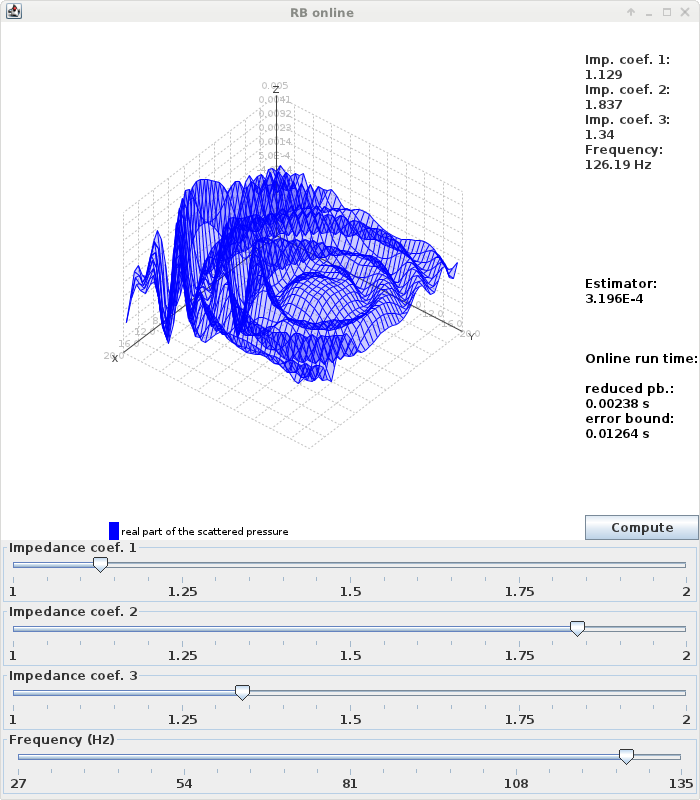}
	\caption{Java applet for the online stage of the RBM for test case 3.
Top panel: total acoustic pressure field on an array of $1681$ points located behind the aircraft.
Bottom panel: selection of the impedance coefficients and of the frequency}
\label{fig:many:javascalable}
\end{figure}

\begin{figure}[htbp]
	\centering
	\includegraphics [width=0.48\textwidth] {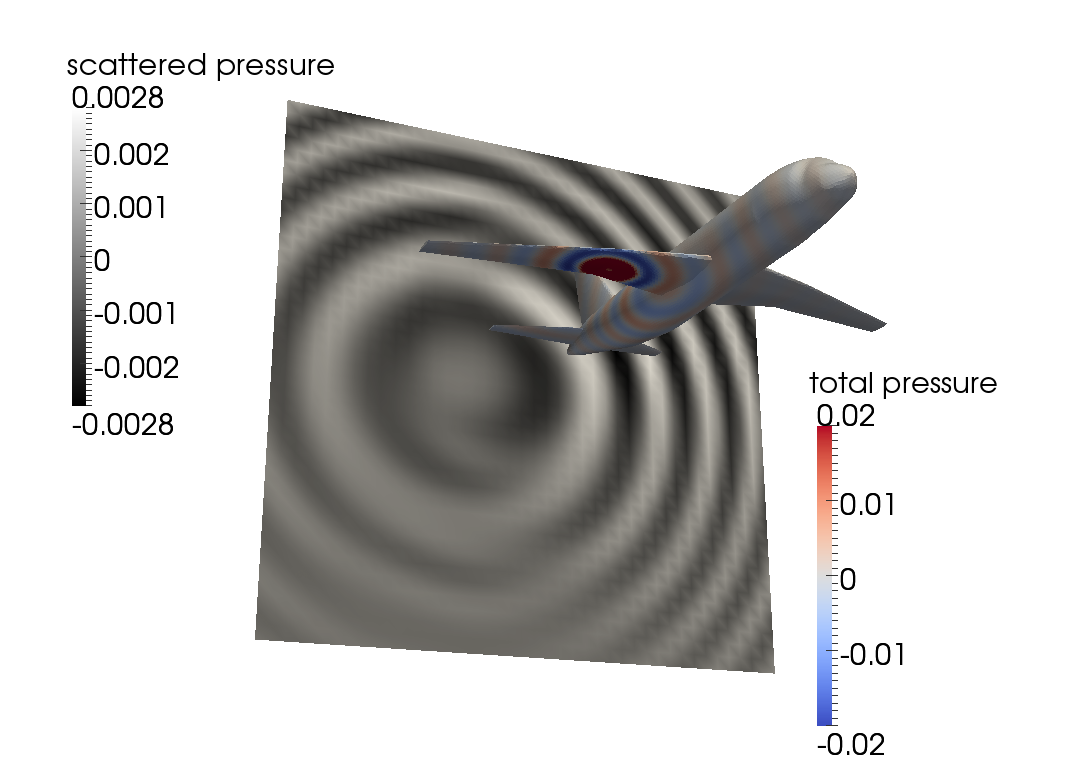}
	\includegraphics [width=0.48\textwidth] {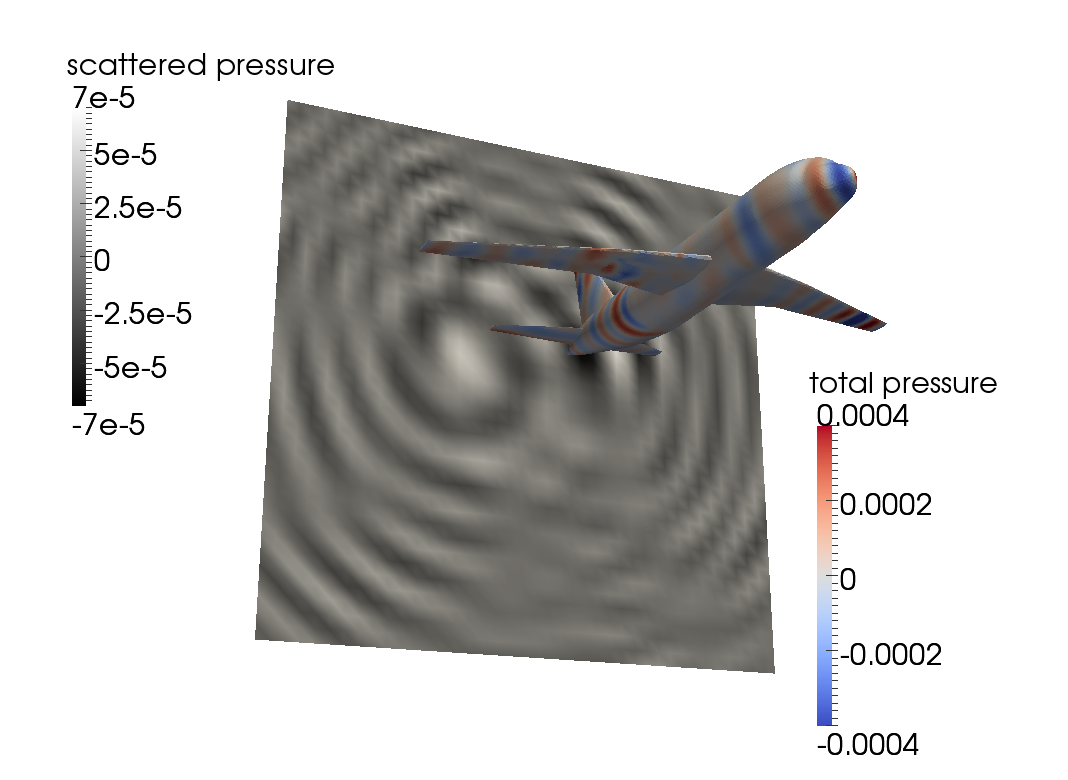}
	\caption{Test case 3. Left: acoustic pressure fields on the aircraft and on an array of $1681$ points located behind the aircraft computed solving the direct problem.
Right: difference between the reduced basis and the direct solution}
\label{fig:many:diffRB}
\end{figure}

\section{Conclusion}
In this work, we derived nonintrusive procedures for the reduced basis method. Their implementation is relatively simple: they have been successfully and easily applied
to the approximation of various matrices and right-hand sides within aeroacoustic simulations. In particular, these procedures allow for the direct use of advanced linear algebra tools, since we are
only dealing with quantities already assembled by the computational code at hand.

\section*{Acknowledgement}
This work was partially supported by EADS Innovation Works.
The authors are thankful to Anthony Patera (MIT) and Guillaume Sylvand (EADS Innovation Works) for fruitful discussions.


\begin{thebibliography}{10}

\bibitem{Barrault}
M.~Barrault, Y.~Maday, N.~C. Nguyen, and A.~T. Patera.
\newblock An 'empirical interpolation' method: application to efficient
  reduced-basis discretization of partial differential equations.
\newblock {\em Comptes Rendus Mathematique}, 339(9):667 -- 672, 2004.

\bibitem{precomp}
B.~Carpentieri, I.~Duff, L.~Giraud, and G.~Sylvand.
\newblock Combining fast multipole techniques and an approximate inverse
  preconditioner for large electromagnetism calculations.
\newblock {\em SIAM Journal on Scientific Computing}, 27(3):774--792, 2005.

\bibitem{casenavephd}
F.~Casenave.
\newblock {\em Reduced order methods applied to aeroacoustic problems solved by
  integral equations}.
\newblock PhD thesis, Universit\'e Paris-Est, 2013.

\bibitem{M2AN}
F.~Casenave, A.~Ern, and T.~Leli\`evre.
\newblock Accurate and online-efficient evaluation of the a posteriori error
  bound in the reduced basis method.
\newblock {\em ESAIM: Math. Model. Numer. Anal.}, 48:207--229, 2014.

\bibitem{jcp}
F.~Casenave, A.~Ern, and G.~Sylvand.
\newblock Coupled {BEM-FEM} for the convected {H}elmholtz equation with
  non-uniform flow in a bounded domain.
\newblock {\em J. Comput. Phys.}, 257, Part A:627--644, 2014.

\bibitem{RBconv2}
R.~A. DeVore, G.~Petrova, and P.~Wojtaszczyk.
\newblock Greedy algorithms for reduced bases in {B}anach spaces.
\newblock {\em Constructive Approximation}, 37(3):455--466, 2013.

\bibitem{Huynh}
D.~B.~P. Huynh, A.~T. Patera, G.~Rozza, and S.~Sen.
\newblock A successive constraint linear optimization method for lower bounds
  of parametric coercivity and inf-sup stability constants.
\newblock {\em Comptes Rendus Mathematique}, 345(8):473 -- 478, 2007.

\bibitem{Machiels}
I.~B. Oliveira A.~T.~Patera L.~Machiels, Y.~Maday and D.~V. Rovas.
\newblock Output bounds for reduced-basis approximations of symmetric positive
  definite eigenvalue problems.
\newblock {\em C. R. Acad. Sci. Paris, Ser. I}, 331, 2005.

\bibitem{Rozzahelm}
T.~Lassila, A.~Manzoni, and G.~Rozza.
\newblock On the approximation of stability factors for general parametrized
  partial differential equations with a two-level affine decomposition.
\newblock {\em ESAIM: Math. Model. Numer. Anal.}, 46:1555--1576, 11 2012.

\bibitem{RB}
L.~Machiels, Y.~Maday, A.~T. Patera, C.~Prud\textquoteright~homme, D.~V. Rovas,
  G.~Turinici, and K.~Veroy.
\newblock Reliable real-time solution of parametrized partial differential
  equations: Reduced-basis output bound methods.
\newblock {\em CJ Fluids Engineering}, 124:70--80, 2002.

\bibitem{Maday}
Y.~Maday, N.~C. Nguyen, A.~T. Patera, and S.~Pau.
\newblock A general multipurpose interpolation procedure: the magic points.
\newblock {\em Communications On Pure And Applied Analysis}, 8(1):383--404,
  2008.

\bibitem{poweriteration}
R.~V. Mises and H.~Pollaczek-Geiringer.
\newblock Praktische verfahren der gleichungsaufl\"osung.
\newblock {\em ZAMM - Journal of Applied Mathematics and Mechanics /
  Zeitschrift f\"ur Angewandte Mathematik und Mechanik}, 9(1):58--77, 1929.

\bibitem{prud2}
A.~T. Patera, C.~Prud\textquoteright homme, D.~V. Rovas, and K.~Veroy.
\newblock A posteriori error bounds for reduced-basis approximation of
  parametrized noncoercive and nonlinear elliptic partial differential
  equations.
\newblock {\em Proceedings of the 16th AIAA Computational Fluid Dynamics
  Conference}, 2003.

\bibitem{sylvand}
G~Sylvand.
\newblock {\em La m\'ethode multip\^ole rapide en \'electromagn\'etisme :
  Performances, parall\'elisation, applications}.
\newblock PhD thesis, Universit\'e de Nice-Sophia Antipolis, 2002.

\end{thebibliography}
\end{document}